\documentclass[11pt]{article}
\usepackage[centertags]{amsmath}
\usepackage{amsfonts}
\usepackage{amssymb}
\usepackage{amsthm}
\usepackage{newlfont}
\usepackage{graphicx}
\usepackage{latexcad}

\usepackage[center]{caption2}
 \usepackage {longtable}
 \newtheorem{thm}{Theorem}[section]
\newtheorem{cor}[thm]{Corollary}
\newtheorem{lem}[thm]{Lemma}

\theoremstyle{definition}

\allowdisplaybreaks
\usepackage{amsthm}

\usepackage{titlesec}

\titleformat{\section}
{\centering\normalfont\Large\bfseries}
{\thesection}
{1em}
{}
\titleformat{name=\section,numberless}
{\raggedright\normalfont\Large\bfseries}
{}
{0pt}
{}
\begin{document}

\title{Parallel covering a rhombus with equilateral triangles}
\author{Jingjing Wang and Yanxun Chang\footnote{Corresponding Author.}
\\{\small \texttt{21118010@bjtu.edu.cn; yxchang@bjtu.edu.cn}}\\
School of Mathematics and Statistics,\\ Beijing Jiaotong University, Beijing 100044, P.R. China.}

\date{}
 \maketitle
\begin{abstract}
Suppose that ${R}^{\alpha}$ is a rhombus with side length $1$ and with an interior angle $\alpha$, where
$0<\alpha\leq \frac{\pi}{2}$.
Let $\triangle$ be an equilateral triangle with a side parallel to a side of ${R}^{\alpha}$ and let $\{\triangle_{n}\}$ be a collection of homothetic copies of $\triangle$.
In this paper, we show the following two results:
if $0<\alpha\leq\frac{\pi}{3}$ and the sum of the areas of equilateral triangles from $\{\triangle_{n}\}$ is at least $\frac{\sqrt{3}}{4}(1+\cos\alpha+\frac{\sqrt{3}}{3}\sin\alpha)^{2}$, then these equilateral triangles can parallel cover the rhombus ${R}^{\alpha}$;
if $\frac{\pi}{3}<\alpha\leq\frac{\pi}{2}$ and the sum of the areas of equilateral triangles from $\{\triangle_{n}\}$ is at least $\frac{\sqrt{3}}{4}(1+\frac{2\sqrt{3}}{3}\sin\alpha)^{2}$, then they can parallel cover the rhombus ${R}^{\alpha}$.
Furthermore, these bounds are optimal on their respective intervals.


\end{abstract}

\section{Introduction}
\hspace{1.3em}
Let $P$, $C$, and $C_{n}$ $(n=1,2,\ldots)$ be convex bodies in the plane.
Let $A(P)$ denote the area of $P$.
We recall the following definitions from~\cite{Jan1}.
We say that the collection $\{C_{n}\}$ permits a \emph{covering} of $P$ if there exist rigid motions $\sigma_{n}$ such that $P \subseteq \bigcup\sigma_{n}C_{n}$.
If each $\sigma_{n}$ is a translation, then $\{C_{n}\}$ permits a \emph{translative covering} of $P$.
When $P$ and all $C_{n}$ are polygons, a covering is called \emph{parallel} if there exists a side of $P$ such that each $\sigma_{n} C_{n}$ has a side parallel to it.
Such a side of $P$ is called a \emph{base} of $P$.
Let $f(P, C)$ be the smallest positive number such that any collection of positive homothetic copies of $C$ whose total area is not less than $f(P, C)\cdot A(P)$ permits a translative covering of $P$.

It is known that $f(S, S) = 3$ for any square $S$~\cite{J.W}.
It is shown in~\cite{Jan1} that $f(T_{e}, S) = 2\sqrt{3}$ for an equilateral triangle $T_{e}$ and a square $S$.
For parallel coverings by squares, results for isosceles triangles are given in~\cite{survey1}, while those for a rhombus, an obtuse triangle and a parallelogram are given in~\cite{Su1,Su2,Su2-1}, respectively.
Results on covering a triangle with triangles can be found in~\cite{redi,Jan2}.
A result on covering a square with isosceles right triangles can be found in~\cite{Su3}.

In this paper, we consider the problem of parallel covering a rhombus with equilateral triangles.
Throughout the paper, the admissible rigid motions are translations and rotations (denoted by $\lambda_{n}$).
We still refer to such coverings as parallel coverings.
Let ${R}^{\alpha}$ be a rhombus with side length $1$ and with an interior angle $\alpha$, where $0<\alpha\leq \frac{\pi}{2}$.
Let $\triangle$ be an equilateral triangle with a side parallel to a side of ${R}^{\alpha}$ and let $\{\triangle_{n}\}$ be a collection of homothetic copies of $\triangle$.
Let $\varrho({R}^{\alpha}, \triangle)$ be the smallest positive number such that any collection of homothetic copies of $\triangle$ with total area not less than $\varrho({R}^{\alpha},\triangle)\cdot A(R^{\alpha})$ permits a parallel covering of ${R}^{\alpha}$.
The present problem differs from the related square-covering problems in that the interaction between the $60^\circ$ geometry of equilateral triangles and the angle $\alpha$ of the rhombus leads to different
residual configurations.

The aim of this paper is to determine the value of $ \varrho({R}^{\alpha}, \triangle)$.
More precisely, we prove that if $0<\alpha\leq\frac{\pi}{3}$, then
$\varrho({R}^{\alpha}, \triangle)=\frac{\sqrt{3} (\sqrt{3}+\sqrt{3}\cos\alpha+\sin\alpha)^{2}}{12\sin\alpha}$ (see Theorem~\ref{thm:1});
if $\frac{\pi}{3}<\alpha\leq\frac{\pi}{2}$, then
$\varrho({R}^{\alpha}, \triangle)=\frac{\sqrt{3} (\sqrt{3}+2\sin\alpha)^{2}}{12\sin\alpha}$ (see Theorems~\ref{thm:2}--\ref{thm:3}).

\section{Preliminaries}
\hspace{1.3em}
The line segment connecting points $A$ and $B$ is denoted by $AB$, and its length is denoted by $|AB|$.
The line through points $A$ and $B$ is denoted by $\overleftrightarrow{AB}$.
To prove our main result, we need the following result from~\cite{redi}.

\begin{lem}\label{lem:1}
Let $\triangle$ be a triangle and let $\{\triangle_{n}\}$ be any collection of positive homothetic copies of $\triangle$.
If $\sum A(\triangle_{n})\ge2\cdot A(\triangle)$, then $\{\triangle_{n}\}$ permits a parallel covering of $\triangle$.
\end{lem}

\begin{lem}\label{lem:2}
Let $Z_{1} := ABCD$ be a trapezoid with $AB \parallel CD$, $|CD| = a > 0$, height $h > 0$, base angles $\alpha \in (0, \frac{\pi}{2}]$ and $\beta = \frac{\pi}{3}$ \textup{(see Fig.~\ref{Fig.1})}.
Then $|AB|=a+\frac{\sqrt{3}}{3}h+h\cot\alpha$.
Let $\triangle$ be an equilateral triangle with a side parallel to $AB$ and let $\{\triangle_{n}\}$ be a collection of homothetic copies of $\triangle$.
Let $a_{n}$ denote the side length of $\triangle_{n}$ for $n = 1, 2, \ldots$.
Without loss of generality we may assume that $a_{1} \geq a_{2} \geq \cdots$.
If
	\begin{align*}
	\sum A(\triangle_{n}) \geq
	\begin{cases}
	\frac{1}{2}h(2a+\frac{\sqrt{3}}{3}h+h\cot\alpha)+\frac{\sqrt{3}}{2}a_{1}(a+\frac{\sqrt{3}}{3}h+h\cot\alpha),
	&\text{if }\alpha\in(0, \frac{\pi}{3}],\\
	\frac{1}{2}h(2a+\frac{\sqrt{3}}{3}h+h\cot\alpha)+ \frac{\sqrt{3}}{2}a_{1}a +\frac{\sqrt{3}}{4}a_{1}h(\sqrt{3}+\cot\alpha),
	& \text{if } \alpha\in(\frac{\pi}{3}, \frac{\pi}{2}],
	\end{cases}
	\end{align*}
then $\{\triangle_{n}\}$ permits a parallel covering of $Z_{1}$.
\begin{figure}[htb]
		\centering
		\includegraphics[width=15cm]{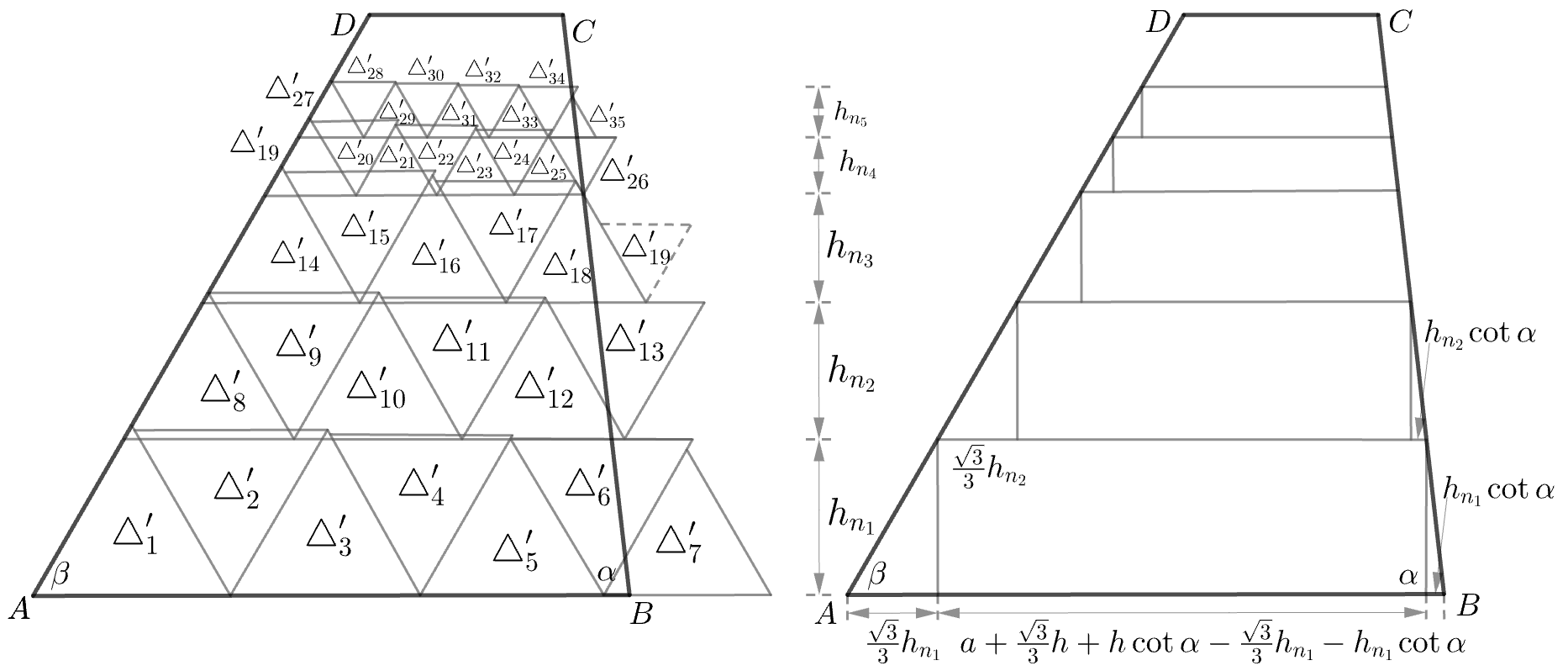}
		\caption{A covering method of $Z_{1}$}
		\label{Fig.1}
\end{figure}
\end{lem}

\begin{proof}
We use the following procedure to construct a parallel covering of $Z_{1}$ with equilateral triangles, following the method of~\cite{Jan1, J.W} (see Fig.~\ref{Fig.1}).
We place the equilateral triangles layer by layer from left to right, with one side of each triangle parallel to $AB$.
In particular, the base of the first layer contains the segment $AB$.

In each layer, the first triangle is positioned so that one of its sides lies along $AD$ and another side lies along the base of the layer.
Each subsequent triangle is placed to the right of the preceding one so
that one of its sides is contained in a side of the preceding triangle,
with the two sides having a common endpoint on the base of the layer,
and the interiors of the triangles are pairwise disjoint.

A layer is called \emph{full} if there is at least one unused triangle
in the collection and the first unused triangle in the given order
cannot be placed, according to the above placement rule, to the right
of the last triangle already placed, in such a way that its interior
has a nonempty intersection with $Z_{1}$.
If the collection is finite and all its triangles have been used before
$Z_1$ is completely covered, the construction stops. We then classify
the last layer by considering a hypothetical equilateral triangle whose
side length does not exceed that of the last triangle placed.
If no such triangle can be placed in the same way, the last layer is regarded as
\emph{full}; otherwise, it is the final non-full layer.

For each $n$, let $\lambda_{n}$ be a rigid motion used to place $\triangle_{n}$, and define $\triangle'_{n} := \lambda_{n}(\triangle_{n})$.
For simplicity, each placed triangle is labeled $\triangle_n'$ in all figures.

For the $j$-th full layer, let $P_{j}$ be the unique intersection point of the base of this layer with the side $BC$ of $Z_{1}$, and let $\triangle'_{p_{j}}$ denote the first triangle in the left-to-right placement order whose boundary contains $P_{j}$.
To determine the last triangle $\triangle'_{n_{j}}$ of the $j$-th full layer and the corresponding layer height $h_{n_{j}}$, we first check whether $\triangle'_{p_{j}}$ is the first triangle in that layer.

If $\triangle'_{p_{j}}$ is not the first triangle in that layer, we consider the following three cases according to the position of the top vertex of $\triangle'_{p_{j}}$ relative to the line $\overleftrightarrow{BC}$.

\begin{enumerate}
		\item\label{case:right} \textit{
		The top vertex of $\triangle'_{p_{j}}$ lies strictly to the right of the line $\overleftrightarrow{BC}$}.
		In this case, $\triangle'_{p_{j}}$ is the last triangle of the layer.
		We set $n_{j}=p_{j}$ and $h_{n_{j}}=\frac{\sqrt{3}}{2}a_{p_{j}-1}$ (the height of $\triangle'_{p_{j}-1}$).
		For example, in Fig.~\ref{Fig.1}, $\triangle'_{7}$ has its top vertex strictly to the right of the line $\overleftrightarrow{BC}$.
		Hence, $\triangle'_{n_{1}} = \triangle'_{7}$ and $h_{n_{1}} = \frac{\sqrt{3}}{2}a_{6}$.
		\item\label{case:left} \textit{
		The top vertex of $\triangle'_{p_{j}}$ lies strictly to the left of the line $\overleftrightarrow{BC}$}.
		If $\frac{\sqrt{3}}{2}a_{p_{j}+1}>|\triangle'_{p_{j}} \cap BC|\cdot\sin\alpha$, then $\triangle'_{p_{j}+1}$ is the last triangle of the layer.
		We set $n_{j}=p_{j}+1$ and $h_{n_{j}}=\frac{\sqrt{3}}{2}a_{p_{j}+1}$.
		Otherwise, the layer ends with $\triangle'_{p_{j}}$. We set $n_{j}=p_{j}$ and
		$h_{n_{j}}=|\triangle'_{p_{j}}\cap BC|\cdot\sin\alpha$.
		For example, in the second layer of Fig.~\ref{Fig.1}, the top vertex of $\triangle'_{12}$ lies strictly to the left of the line $\overleftrightarrow{BC}$ and $\frac{\sqrt{3}}{2}a_{13}>|\triangle'_{12}\cap BC|\cdot\sin\alpha$.
		Therefore, $\triangle'_{n_{2}}=\triangle'_{13}$ and $h_{n_{2}}=\frac{\sqrt{3}}{2}a_{13}$.
		In the third layer, the opposite inequality holds for $\triangle'_{18}$, so $\triangle'_{n_{3}}=\triangle'_{18}$ and $h_{n_{3}}=|\triangle'_{18}\cap BC|\cdot\sin\alpha$.
		\item \textit{
			The top vertex of $\triangle'_{p_{j}}$ lies on the line $\overleftrightarrow{BC}$}.
		As in item~\ref{case:right}, $\triangle'_{p_{j}}$ is the last triangle of the layer. We set $n_{j}=p_{j}$ and $h_{n_{j}}=\frac{\sqrt{3}}{2}a_{p_{j}-1}$.
		For example, this occurs for $\triangle'_{35}$ in the fifth layer of Fig.~\ref{Fig.1}, and hence $\triangle'_{n_{5}} = \triangle'_{35}$ and $h_{n_{5}} = \frac{\sqrt{3}}{2}a_{34}$.
\end{enumerate}

It remains to consider the case in which $\triangle'_{p_{j}}$ is the first triangle in that layer.

If the top vertex of $\triangle'_{p_{j}}$ lies on or strictly to the right of the line $\overleftrightarrow{BC}$, then the covering is complete, since $\triangle'_{p_{j}}$ already covers the remaining uncovered part of $Z_{1}$.
If the top vertex of $\triangle'_{p_{j}}$ lies strictly to the left of the line $\overleftrightarrow{BC}$, then the covering is complete if $\triangle'_{p_{j}}$ covers the remaining uncovered part of $Z_{1}$;
otherwise, $\triangle'_{n_{j}}$ and $h_{n_{j}}$ are determined by the rule in item~\ref{case:left}.

Once the last triangle $\triangle'_{n_j}$ of the $j$-th full layer
and the corresponding height $h_{n_j}$ have been determined, if the
construction proceeds to a new layer, its base is placed at a vertical
distance $h_{n_j}$ above the base of the $j$-th full layer. The first
triangle of the new layer is $\triangle'_{n_j+1}$, with one side lying
along $AD$ and another along the base of the new layer.
This process is repeated until $Z_{1}$ is completely covered or all triangles in the collection have been used.

Contrary to the statement, we suppose that it is impossible to cover $Z_{1}$ with the equilateral triangles $\triangle_{1}, \triangle_{2}, \ldots$ by this method.

We first consider the case where finitely many full layers are created.
Let $k$ denote the number of full layers.
If $k=0$, then all triangles are placed in the first layer.
Hence,
$\sum A(\triangle_n)\leq\frac{\sqrt3}{2}a_1(a+\frac{\sqrt3}{3}h+h\cot\alpha)$.
This leads to a contradiction.
Therefore, we may assume that $k\geq1$.

By construction, the first triangle in the $(j+1)$-th layer is $\triangle_{n_{j}+1}'$.
By the definition of $h_{n_{j}}$, we have $\frac{\sqrt{3}}{2} a_{n_{j}+1} \le h_{n_{j}}$ whenever the $(j+1)$-th layer exists.
Clearly, $\sum_{j=1}^k h_{n_{j}} < h$.

The total area of the equilateral triangles lying in the first layer is less than $\frac{\sqrt{3}}{2}a_{1}(a+\frac{\sqrt{3}}{3}h+h\cot\alpha)+\frac{\sqrt{3}}{3}h_{n_{1}}^{2}$, where $\frac{\sqrt{3}}{2}a_{1}(a+\frac{\sqrt{3}}{3}h+h\cot\alpha)$ is the area of a parallelogram with base length $a+\frac{\sqrt{3}}{3}h+h\cot\alpha$ and height $\frac{\sqrt{3}}{2}a_{1}$, and $\frac{\sqrt{3}}{3}h_{n_{1}}^{2}$ is the area of an equilateral triangle with height $h_{n_{1}}$.

Similarly, using $\frac{\sqrt{3}}{2}a_{n_{j-1}+1}\leq h_{n_{j-1}}$, we know that the total area of the equilateral triangles lying in the $j$-th layer, for $j\in\{2,3,\ldots,k\}$, is less than
\begin{align*}
		(a+\frac{\sqrt{3}}{3}h+h\cot\alpha - \sum\nolimits_{i=1}^{j-1} h_{n_{i}}(\frac{\sqrt{3}}{3} + \cot\alpha))h_{n_{j-1}} +\frac{\sqrt{3}}{3}h_{n_{j}}^{2}.
\end{align*}

Moreover, the total area of the equilateral triangles lying in the $(k + 1)$-th layer, if it has been created, does not exceed
\begin{align*}
		(a+\frac{\sqrt{3}}{3}h+h\cot\alpha - \sum\nolimits_{i=1}^{k}h_{n_{i}}(\frac{\sqrt{3}}{3} + \cot\alpha))h_{n_{k}}.
\end{align*}

Observe that the sum
\begin{align*}
		\sum\nolimits_{j=1}^{k} [ (a+\frac{\sqrt{3}}{3}h+h\cot\alpha - \sum\nolimits_{i=1}^{j}h_{n_{i}} (\frac{\sqrt{3}}{3}+\cot\alpha)) h_{n_{j}} + \frac{\sqrt{3}}{6} h_{n_{j}}^2 + \frac{1}{2} h_{n_{j}}^2 \cot\alpha]
\end{align*}
does not exceed the area of a trapezoid with base lengths $a$ and $a+\frac{\sqrt{3}}{3}h+h \cot\alpha$ and with height $h$, that is, does not exceed $\frac{1}{2}h(2a+\frac{\sqrt{3}}{3}h+h\cot\alpha)$.
Hence,
\begin{align*}
		\sum A(\triangle_{n})
		&<\frac{\sqrt{3}}{2}a_{1}(a+\frac{\sqrt{3}}{3}h+h\cot\alpha)+\frac{\sqrt{3}}{3}h_{n_{1}}^{2}\\
		&~~~+\sum\nolimits_{j=2}^{k}[(a+\frac{\sqrt{3}}{3}h+h\cot\alpha - \sum\nolimits_{i=1}^{j-1} h_{n_{i}}(\frac{\sqrt{3}}{3} + \cot\alpha))h_{n_{j-1}}+ \frac{\sqrt{3}}{3}h_{n_{j}}^{2}]\\
		&~~~+(a+\frac{\sqrt{3}}{3}h+h\cot\alpha - \sum\nolimits_{i=1}^{k}h_{n_{i}}(\frac{\sqrt{3}}{3} + \cot\alpha))h_{n_{k}} \\
		&=\sum\nolimits_{j=1}^{k} [ (a+\frac{\sqrt{3}}{3}h+h\cot\alpha - \sum\nolimits_{i=1}^{j}h_{n_{i}} (\frac{\sqrt{3}}{3}+\cot\alpha)) h_{n_{j}} +\frac{\sqrt{3}}{6} h_{n_{j}}^2 \\
		&~~~+ \frac{1}{2} h_{n_{j}}^2 \cot\alpha]+\frac{\sqrt{3}}{2}a_{1}(a+\frac{\sqrt{3}}{3}h+h\cot\alpha)+\sum\nolimits_{j=1}^{k}h_{n_{j}}^{2}(\frac{\sqrt{3}}{6} - \frac{1}{2}\cot\alpha)\\
		&<\frac{1}{2}h(2a+\frac{\sqrt{3}}{3}h+h\cot\alpha)+\frac{\sqrt{3}}{2}a_{1}(a+\frac{\sqrt{3}}{3}h+h\cot\alpha)\\&~~~+\sum\nolimits_{j=1}^{k}h_{n_{j}}^{2}(\frac{\sqrt{3}}{6} - \frac{1}{2}\cot\alpha).
\end{align*}

From $\sum_{j=1}^{k}h_{n_{j}} < h$ and $\frac{\sqrt{3}}{2}a_{1}\geq h_{n_{1}} \geq h_{n_{2}} \geq \cdots \geq h_{n_{k}}$, we obtain
$\sum\nolimits_{j=1}^{k}h_{n_{j}}^{2}\leq h_{n_{1}}\sum\nolimits_{j=1}^{k}h_{n_{j}}<\frac{\sqrt{3}}{2}a_{1}h$.
Moreover, since \(\alpha\in(0,\frac{\pi}{2}]\), we have
\begin{align*}
		\frac{\sqrt{3}}{6}-\frac{1}{2}\cot\alpha
		\begin{cases}
			\leq0, & \text{if } \alpha\in(0, \frac{\pi}{3}], \\
			>0, & \text{if } \alpha\in(\frac{\pi}{3}, \frac{\pi}{2}].
		\end{cases}
\end{align*}

Thus,
\begin{align*}
	\sum A(\triangle_{n}) <
	\begin{cases}
	\frac{1}{2}h(2a+\frac{\sqrt{3}}{3}h+h\cot\alpha)+	 \frac{\sqrt{3}}{2}a_{1}(a+\frac{\sqrt{3}}{3}h+h\cot\alpha),
	& \text{if } \alpha\in(0, \frac{\pi}{3}], \\
	\frac{1}{2}h(2a+\frac{\sqrt{3}}{3}h+h\cot\alpha)+ \frac{\sqrt{3}}{2}a_{1}a +\frac{\sqrt{3}}{4}a_{1}h(\sqrt{3}+\cot\alpha), & \text{if } \alpha\in(\frac{\pi}{3}, \frac{\pi}{2}],
   \end{cases}
\end{align*}
which is a contradiction.

Now consider the case where infinitely many full layers are created.
Using the strict estimates obtained above together with
\begin{align*}
	&\sum\nolimits_{j=1}^{\infty}[(a+\frac{\sqrt{3}}{3}h+h\cot\alpha-\sum\nolimits_{i=1}^{j}h_{n_{i}}(\frac{\sqrt{3}}{3}+\cot\alpha))h_{n_{j}}
	+\frac{\sqrt{3}}{6}h_{n_{j}}^2+\frac12 h_{n_{j}}^2\cot\alpha]\\
	&~~~\le\frac{1}{2}h(2a+\frac{\sqrt{3}}{3}h+h\cot\alpha),
\end{align*}
we obtain
	\[
	\sum A(\triangle_{n})
	<
	\frac12h(2a+\frac{\sqrt{3}}{3}h+h\cot\alpha)+\frac{\sqrt{3}}{2}a_{1}(a+\frac{\sqrt{3}}{3}h+h\cot\alpha)+\sum\nolimits_{j=1}^{\infty}h_{n_{j}}^2(\frac{\sqrt{3}}{6}-\frac12\cot\alpha).
	\]
Since $\sum\nolimits_{j=1}^{\infty} h_{n_{j}} \leq h$ and $\sum\nolimits_{j=1}^{\infty} h_{n_{j}}^{2} <h_{n_{1}}\sum\nolimits_{j=1}^{\infty} h_{n_{j}} \leq \frac{\sqrt{3}}{2}a_{1}h$, using again the sign of $\frac{\sqrt{3}}{6}-\frac12\cot\alpha$, we obtain the same contradiction.
\end{proof}

Corollaries~\ref{cor:3} and~\ref{cor:5} below are direct applications
of Lemma~\ref{lem:2}, whereas Corollaries~\ref{cor:4}
and~\ref{cor:6} follow from the same layer-by-layer construction in
the degenerate case $a=0$.
In Corollaries~\ref{cor:3}--\ref{cor:6}, let $\triangle$ be an equilateral triangle with one side parallel to $AB$, and let $\{\triangle_{n}\}$ be a collection of homothetic copies of $\triangle$ with side lengths $a_{n}$.
All coverings considered below are parallel.
Without loss of generality we may assume that $a_{1}\ge a_{2}\ge\cdots$.

\begin{cor} \label{cor:3}
Let $Z_{1}:=ABCD$ be the trapezoid described in Lemma~\ref{lem:2} \textup{(see Fig.~\ref{Fig.1})}.
If
\begin{align*}
	\sum A(\triangle_{n}) \geq
	\begin{cases}
	\frac{1}{2}h(2a+\frac{\sqrt{3}}{3}h+h\cot\alpha)+	 \frac{\sqrt{3}}{2}a_{1}(a+\frac{\sqrt{3}}{3}h+h\cot\alpha),
	& \text{if } \alpha\in(0, \frac{\pi}{3}], \\
	\frac{1}{2}h(2a+\frac{\sqrt{3}}{3}h+h\cot\alpha)+ \frac{\sqrt{3}}{2}a_{1}a +\frac{\sqrt{3}}{4}a_{1}h(\sqrt{3}+\cot\alpha), & \text{if } \alpha\in(\frac{\pi}{3}, \frac{\pi}{2}],
	\end{cases}
\end{align*}
then $\{\triangle_{n}\}$ permits a parallel covering of $Z_{1}$.
Moreover, if $k$ is the smallest integer such that $\triangle_{1}$, $\triangle_{2}$, $\ldots$, $\triangle_{k}$ permit a covering of $Z_{1}$ \textup{(see Fig.~\ref{Fig.1})}, then
\begin{align*}
\sum\nolimits_{n=1}^{k} A(\triangle_{n}) <
\begin{cases}
\frac{1}{2}h(2a+\frac{\sqrt{3}}{3}h+h\cot\alpha)+\frac{\sqrt{3}}{2}a_{1}(a+\frac{\sqrt{3}}{3}h+h\cot\alpha)\\+A(\triangle_{k}),
& \text{if } \alpha\in(0, \frac{\pi}{3}],\\
\frac{1}{2}h(2a+\frac{\sqrt{3}}{3}h+h\cot\alpha)+ \frac{\sqrt{3}}{2}a_{1}a +\frac{\sqrt{3}}{4}a_{1}h(\sqrt{3}+\cot\alpha)\\
+A(\triangle_{k}),
& \text{if } \alpha\in(\frac{\pi}{3}, \frac{\pi}{2}].
\end{cases}
\end{align*}
\end{cor}

\begin{cor} \label{cor:4}
Let $T_{1}:=ABC$ be a triangle with base $AB$, height $h>0$ (from $C$ to $AB$), and base angles $\alpha\in(0, \frac{\pi}{2}]$ and $\beta=\frac{\pi}{3}$ \textup{(see Fig.~\ref{Fig.2})}.
If
\begin{align*}
	\sum A(\triangle_{n}) \ge
	\begin{cases}
	\frac{1}{2}h^{2}(\frac{\sqrt{3}}{3}+\cot\alpha)+\frac{\sqrt{3}}{2}a_{1}h(\frac{\sqrt{3}}{3}+\cot\alpha),
	& \text{if } \alpha\in(0, \frac{\pi}{3}], \\
	\frac{1}{2}h^{2}(\frac{\sqrt{3}}{3}+\cot\alpha)+\frac{\sqrt{3}}{4}a_{1}h(\sqrt{3}+\cot\alpha),
	& \text{if } \alpha\in(\frac{\pi}{3}, \frac{\pi}{2}],
	\end{cases}
\end{align*}
then $\{\triangle_{n}\}$ permits a parallel covering of $T_{1}$.
Moreover, if $k$ is the smallest integer such that $\triangle_{1}$, $\triangle_{2}$, $\ldots$, $\triangle_{k}$ permit a covering of $T_{1}$, then
\begin{align*}
	\sum\nolimits_{n=1}^{k} A(\triangle_{n}) <
	\begin{cases}
	\frac{1}{2}h^{2}(\frac{\sqrt{3}}{3}+\cot\alpha)+\frac{\sqrt{3}}{2}a_{1}h(\frac{\sqrt{3}}{3}+\cot\alpha)+ A(\triangle_{k}),
	& \text{if } \alpha\in(0, \frac{\pi}{3}], \\
	\frac{1}{2}h^{2}(\frac{\sqrt{3}}{3}+\cot\alpha)+\frac{\sqrt{3}}{4}a_{1}h(\sqrt{3}+\cot\alpha)+ A(\triangle_{k}),
	& \text{if } \alpha\in(\frac{\pi}{3}, \frac{\pi}{2}].
	\end{cases}
\end{align*}
\begin{figure}[htb]
		\centering
		\includegraphics[width=1.8cm]{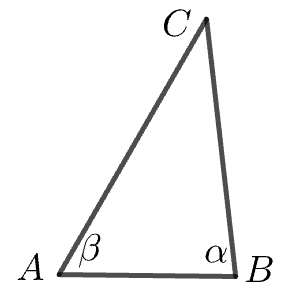}
		\caption{$T_{1}:= ABC$}
		\label{Fig.2}
\end{figure}
\end{cor}

\begin{cor} \label{cor:5}
Let $Z_{2} := ABCD$ be a trapezoid with $AB \parallel CD$, $|CD| = a > 0$, height $h > 0$, base angles $\alpha = \frac{\pi}{2}$ and $\beta = \frac{\pi}{3}$ \textup{(see Fig.~\ref{Fig.3})}.
Then $|AB|=a+\frac{\sqrt{3}}{3}h$.
If $\sum A(\triangle_{n}) \geq \frac{1}{2}h(2a+\frac{\sqrt{3}}{3}h)+ \frac{\sqrt{3}}{2}a_{1}(a+\frac{\sqrt{3}}{2}h)$, then $\{\triangle_{n}\}$ permits a parallel covering of $Z_{2}$.
Moreover, if $k$ is the smallest integer such that $\triangle_{1}, \triangle_{2}, \ldots, \triangle_{k}$ permit a covering of $Z_{2}$, then
$\sum\nolimits_{n=1}^{k} A(\triangle_{n}) <\frac{1}{2}h(2a+\frac{\sqrt{3}}{3}h)+
\frac{\sqrt{3}}{2}a_{1}(a+\frac{\sqrt{3}}{2}h)+A(\triangle_{k})$.
\begin{figure}[htb]
		\centering
		\includegraphics[width=2.6cm]{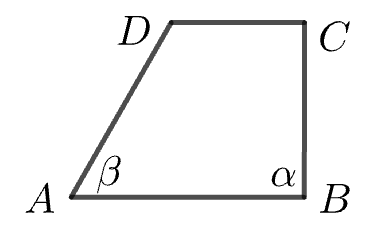}
		\caption{$Z_{2} := ABCD$}
		\label{Fig.3}
	\end{figure}
\end{cor}

\begin{cor} \label{cor:6}
Let $T_{2}:=ABC$ be a triangle with base $AB$, height $h>0$ (from $C$ to $AB$), and base angles $\alpha=\frac{\pi}{2}$ and $\beta=\frac{\pi}{3}$ \textup{(see Fig.~\ref{Fig.4})}.
If $\sum A(\triangle_{n}) \geq\frac{\sqrt{3}}{6}h^{2}+ \frac{3}{4}a_{1}h$, then $\{\triangle_{n}\}$ permits a parallel covering of $T_{2}$.
\begin{figure}[htb]
		\centering
		\includegraphics[width=2.1cm]{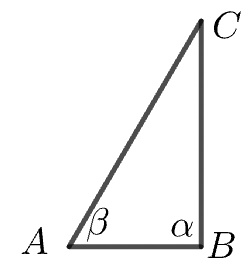}
		\caption{$T_{2}:= ABC$}
		\label{Fig.4}
	\end{figure}
\end{cor}

\section{Main Result}

\hspace{1.3em}
Suppose that ${R}^{\alpha}$ is a rhombus with side length $1$ and with an interior angle $\alpha$, where
$0<\alpha\leq \frac{\pi}{2}$.
Let $\triangle$ be an equilateral triangle with a side parallel to the base of ${R}^{\alpha}$.
Let $\{\triangle_{n}\}$ be a collection of homothetic copies of $\triangle$.
Let $a_{n}$ denote the side length of $\triangle_{n}$ for $n = 1, 2, \ldots$.
Without loss of generality we may assume that $a_{1} \geq a_{2} \geq \cdots$.

\begin{thm}\label{thm:1}
If $0 < \alpha \leq \frac{\pi}{3}$ and
$\sum A( \triangle_{n})\geq\frac{\sqrt{3}}{4}(1+\cos\alpha+\frac{\sqrt{3}}{3}\sin\alpha)^{2}$, then $\{\triangle_{n}\}$ permits a parallel covering of $R^{\alpha}$ and thus
$\varrho({R}^{\alpha}, \triangle)=\frac{\sqrt{3}(\sqrt{3}+\sqrt{3}\cos\alpha+\sin\alpha)^{2}}{12\sin\alpha}$.
\end{thm}

\begin{proof}
We claim that $\{\triangle_{n}\}$ permits a parallel covering of $R^{\alpha}$.
We may assume that $a_{1} < 1 + \cos\alpha+\frac{\sqrt{3}}{3}\sin\alpha$, otherwise $R^{\alpha}$ can be parallel covered by $\triangle_{1}$.
\begin{figure}[!htbp]
	\centering
	\includegraphics[width=10cm]{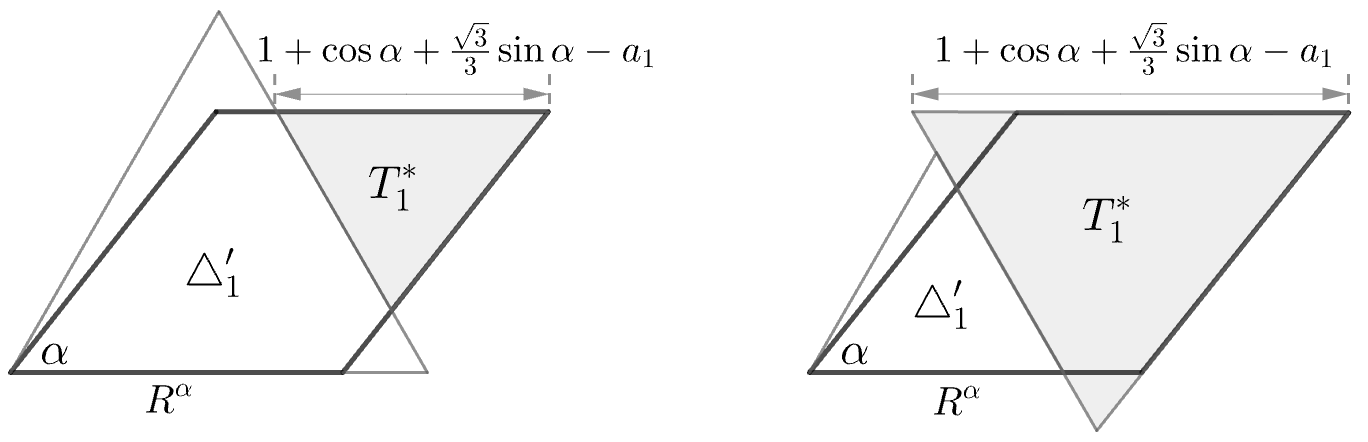}
	\caption{$0<a_{1}<1 + \cos\alpha+\frac{\sqrt{3}}{3}\sin\alpha$}
	\label{Fig.5}
\end{figure}

We place $\triangle_{1}$ as in Fig.~\ref{Fig.5}.
The remaining equilateral triangles are used for the covering of the triangle $T_{1}^{\ast}\supseteq R^{\alpha}\setminus \triangle'_{1}$ with base length $1+\cos\alpha+\frac{\sqrt{3}}{3}\sin\alpha -a_{1}$ and height
$\frac{1+\cos\alpha+\frac{\sqrt{3}}{3}\sin\alpha -a_{1}}{\frac{\sqrt{3}}{3}+\cot\alpha}$.
By Corollary~\ref{cor:4}, we deduce that if $R^{\alpha}$ cannot be covered, then
\begin{align*}
	\sum A(\triangle_{n})
	&< \frac{\sqrt{3}}{4}a_{1}^{2}+\frac{(1+\cos\alpha+\frac{\sqrt{3}}{3}\sin\alpha
		-a_{1})^{2}}{2(\frac{\sqrt{3}}{3}+\cot\alpha)}+ \frac{\sqrt{3}}{2}a_{2}(1+\cos\alpha+\frac{\sqrt{3}}{3}\sin\alpha -a_{1})\\
	&\leq\frac{\sqrt{3}}{4}a_{1}^{2}+\frac{(1+\cos\alpha+\frac{\sqrt{3}}{3}\sin\alpha -a_{1})^{2}}{2(\frac{\sqrt{3}}{3}+\cot\alpha)}+ \frac{\sqrt{3}}{2}a_{1}(1+\cos\alpha+\frac{\sqrt{3}}{3}\sin\alpha -a_{1})\\
	&\leq\frac{\sqrt{3}}{4}(1+\cos\alpha+\frac{\sqrt{3}}{3}\sin\alpha)^2.
\end{align*}
This inequality leads to a contradiction.

By the discussions above we know that $\varrho({R}^{\alpha}, \triangle)\leq\frac{\sqrt{3}(\sqrt{3}+\sqrt{3}\cos\alpha+\sin\alpha)^{2}}{12\sin\alpha}$.
However, any equilateral triangle with side length less than $1+\cos\alpha+\frac{\sqrt{3}}{3}\sin\alpha$ cannot parallel cover ${R}^{\alpha}$, hence
$\varrho({R}^{\alpha}, \triangle)\geq\frac{\sqrt{3}(\sqrt{3}+\sqrt{3}\cos\alpha+\sin\alpha)^{2}}{12\sin\alpha}$.
As a consequence, $\varrho({R}^{\alpha}, \triangle)=\frac{\sqrt{3}(\sqrt{3}+\sqrt{3}\cos\alpha+\sin\alpha)^{2}}{12\sin\alpha}$.
\end{proof}

\begin{thm}\label{thm:2}
If $\frac{\pi}{3} < \alpha \leq \frac{5\pi}{12}$ and
$\sum A(\triangle_{n}) \geq \frac{\sqrt{3}}{4}(1+\frac{2\sqrt{3}}{3}\sin\alpha)^{2}$, then $\{\triangle_{n}\}$ permits a parallel covering of $R^{\alpha}$ and thus $\varrho({R}^{\alpha}, \triangle)=\frac{\sqrt{3}(\sqrt{3}+2\sin\alpha)^{2}}{12\sin\alpha}$.
\end{thm}

\begin{proof}
We claim that $\{\triangle_{n}\}$ permits a parallel covering of $R^{\alpha}$.
We may assume that $a_{1} < 1 + \frac{2\sqrt{3}}{3}\sin\alpha$, otherwise $R^{\alpha}$ can be parallel covered by $\triangle_{1}$.

Throughout the upper-bound estimates below, we fix an arbitrary $\alpha$ in the stated range.
For this fixed $\alpha$, each right-hand side is a quadratic polynomial in the relevant variables $a_{i}$, considered on the corresponding admissible domain.
Its maximum over the closure of that domain is determined by checking the stationary points in the interior and considering all boundary cases.
In each case, we give a point at which the maximum is attained.
The same procedure is used in the proof of Theorem~\ref{thm:3}.

{\textit Case 1}: $1 + \frac{\sqrt{3}}{3}\sin\alpha-\cos\alpha\leq a_{1}<1 + \frac{2\sqrt{3}}{3}\sin\alpha$.

\begin{figure}[htb]
	\centering
    \includegraphics[width=10.5cm]{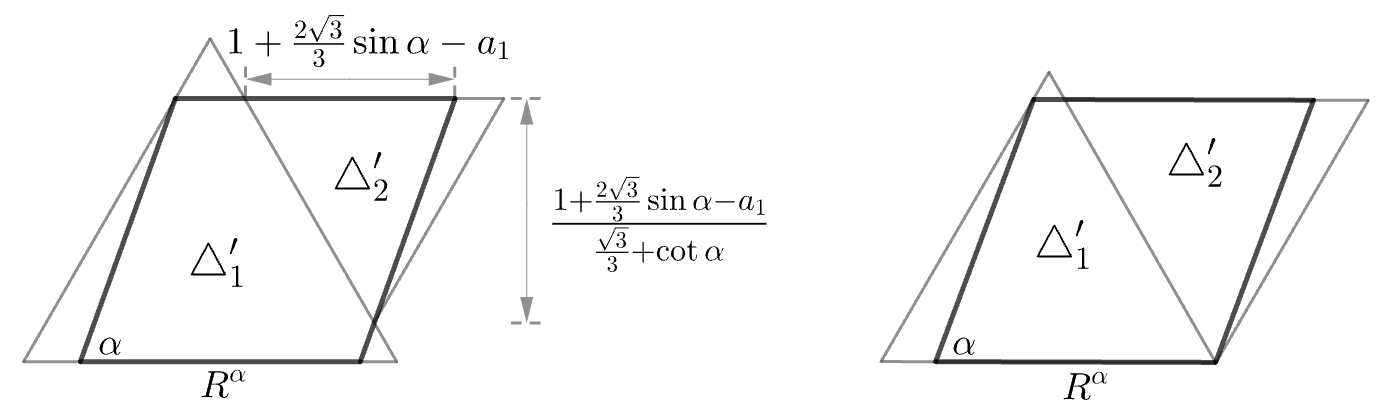}
	\caption{$a_{2}=\frac{2\sqrt{3}}{3}\cdot\frac{1+\frac{2\sqrt{3}}{3}\sin\alpha-a_{1}}{\frac{\sqrt{3}}{3}+\cot\alpha}$}
	\label{Fig.6}
\end{figure}

We assume that $a_{2}<\frac{2\sqrt{3}}{3}\cdot\frac{1+\frac{2\sqrt{3}}{3}\sin\alpha-a_{1}}{\frac{\sqrt{3}}{3}+\cot\alpha}$, otherwise $R^{\alpha}$ can be parallel covered by $\triangle_{1}$ and $\triangle_{2}$ as shown in Fig.~\ref{Fig.6}.

{\textit Subcase 1.1}: $ 1+\frac{2\sqrt{3}}{3}\sin\alpha-a_{1}<a_{2}<\frac{2\sqrt{3}}{3}\cdot\frac{1+\frac{2\sqrt{3}}{3}\sin\alpha-a_{1}}{\frac{\sqrt{3}}{3}+\cot\alpha}$.
\begin{figure}[htb]
	\centering
	\includegraphics[width=12cm]{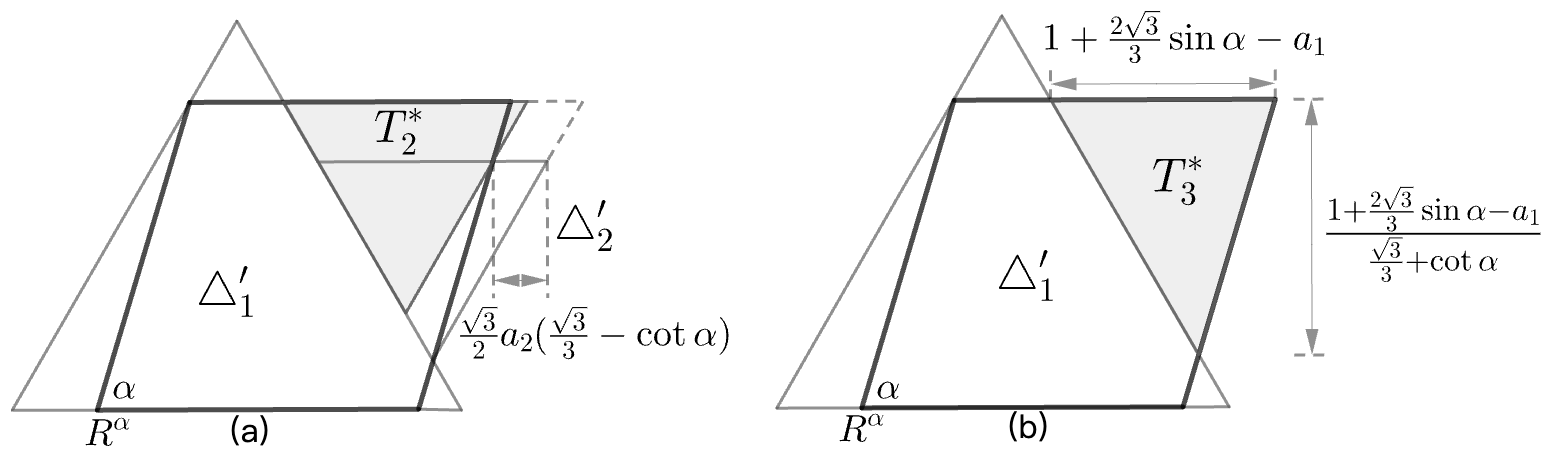}
	\caption{$0<a_{2}<\frac{2\sqrt{3}}{3}\cdot\frac{1+\frac{2\sqrt{3}}{3}\sin\alpha-a_{1}}{\frac{\sqrt{3}}{3}+\cot\alpha}$}
	\label{Fig.7}
\end{figure}

We place $\triangle_{1}$ and $\triangle_{2}$ as in Fig.~\ref{Fig.7}(a).
The remaining equilateral triangles are used for the covering of the equilateral triangle
$T_{2}^{\ast}\supset R^{\alpha}\setminus (\triangle'_{1}\cup\triangle'_{2})$ with side length
$\frac{2\sqrt{3}}{3}\cdot\frac{1+\frac{2\sqrt{3}}{3}\sin\alpha-a_{1}}{\frac{\sqrt{3}}{3}+\cot\alpha}-\frac{\sqrt{3}}{2}a_{2}(\frac{\sqrt{3}}{3}-\cot\alpha)$ and height $\frac{\sqrt{3}}{2}\cdot(\frac{2\sqrt{3}}{3}\cdot\frac{1+\frac{2\sqrt{3}}{3}\sin\alpha-a_{1}}
{\frac{\sqrt{3}}{3}+\cot\alpha}-\frac{\sqrt{3}}{2}a_{2}(\frac{\sqrt{3}}{3}-\cot\alpha))$.
By Lemma~\ref{lem:1}, we deduce that if $R^{\alpha}$ cannot be covered, then
\begin{align*}
	\sum A(\triangle_{n})
	&\!<\! \frac{\sqrt{3}}{4}(a_{1}^{2}+a_{2}^{2})+2\cdot\frac{1}{2}\cdot\frac{\sqrt{3}}{2}\cdot(\frac{2\sqrt{3}}{3}\cdot\frac{1+\frac{2\sqrt{3}}{3}\sin\alpha-a_{1}}{\frac{\sqrt{3}}{3}+\cot\alpha}\!-\!\frac{\sqrt{3}}{2}a_{2}(\frac{\sqrt{3}}{3}-\cot\alpha))^2.
\end{align*}

The maximum of the right-hand side over the closure of the feasible region is $\frac{\sqrt{3}}{4}(1+\frac{2\sqrt{3}}{3}\sin\alpha)^2$ (which is attained at $a_{1}=1+\frac{2\sqrt{3}}{3}\sin\alpha$ and $a_{2}=0$) provided that the feasible region is given by $1+\frac{\sqrt{3}}{3}\sin\alpha-\cos\alpha\le a_{1}<1+\frac{2\sqrt{3}}{3}\sin\alpha$,
$1+\frac{2\sqrt{3}}{3}\sin\alpha-a_{1}<a_{2}<\frac{2\sqrt{3}}{3}\cdot\frac{1+\frac{2\sqrt{3}}{3}\sin\alpha-a_{1}}{\frac{\sqrt{3}}{3}+\cot\alpha}$ and $a_{1}\ge a_{2}$.
This leads to a contradiction.

{\textit Subcase 1.2}: $0<a_{2}\leq 1+\frac{2\sqrt{3}}{3}\sin\alpha-a_{1}$.

We place $\triangle_{1}$ as in Fig.~\ref{Fig.7}(b).
The remaining equilateral triangles are used for the covering of the triangle $T_{3}^{\ast}= R^{\alpha}\setminus \triangle'_{1}$ with base length $1+\frac{2\sqrt{3}}{3}\sin\alpha-a_{1}$ and height
$\frac{1+\frac{2\sqrt{3}}{3}\sin\alpha-a_{1}}{\frac{\sqrt{3}}{3}+\cot\alpha}$.
By Corollary~\ref{cor:4}, we deduce that if $R^{\alpha}$ cannot be covered, then
\begin{align*}
	\sum A(\triangle_{n})
	&< \frac{\sqrt{3}}{4}a_{1}^{2} +\frac{(1+\frac{2\sqrt{3}}{3}\sin\alpha-a_{1})^2}{2(\frac{\sqrt{3}}{3}+\cot\alpha)}
	+\frac{\sqrt{3}}{4}a_{2}\cdot\frac{1+\frac{2\sqrt{3}}{3}\sin\alpha-a_{1}}{\frac{\sqrt{3}}{3}+\cot\alpha}(\sqrt{3}+\cot\alpha)\\
	&\leq\frac{\sqrt{3}}{4}a_{1}^{2} +\frac{(1+\frac{2\sqrt{3}}{3}\sin\alpha-a_{1})^2}{2(\frac{\sqrt{3}}{3}+\cot\alpha)}
	+\frac{\sqrt{3}}{4}\cdot\frac{(1+\frac{2\sqrt{3}}{3}\sin\alpha-a_{1})^2}{\frac{\sqrt{3}}{3}+\cot\alpha}(\sqrt{3}+\cot\alpha)\\
	&=\frac{\sqrt{3}}{4}a_{1}^{2}
	+\frac{1}{4}\cdot\frac{(1+\frac{2\sqrt{3}}{3}\sin\alpha-a_{1})^2}{\frac{\sqrt{3}}{3}+\cot\alpha}(5+\sqrt{3}\cot\alpha).
\end{align*}

The maximum of the right-hand side over the closure of the feasible region is $\frac{\sqrt{3}}{4}(1+\frac{2\sqrt{3}}{3}\sin\alpha)^2$
 (which is attained at $a_{1}=1+\frac{2\sqrt{3}}{3}\sin\alpha$)
provided that the feasible region is given by $1+\frac{\sqrt{3}}{3}\sin\alpha-\cos\alpha\le a_{1}<1+\frac{2\sqrt{3}}{3}\sin\alpha$.
This leads to a contradiction.

{\textit Case 2}: $\frac{2\sqrt{3}}{3}\sin\alpha \leq a_{1}<1 + \frac{\sqrt{3}}{3}\sin\alpha-\cos\alpha$.

\begin{figure}[htb]
	\centering
	\includegraphics[width=14.5cm]{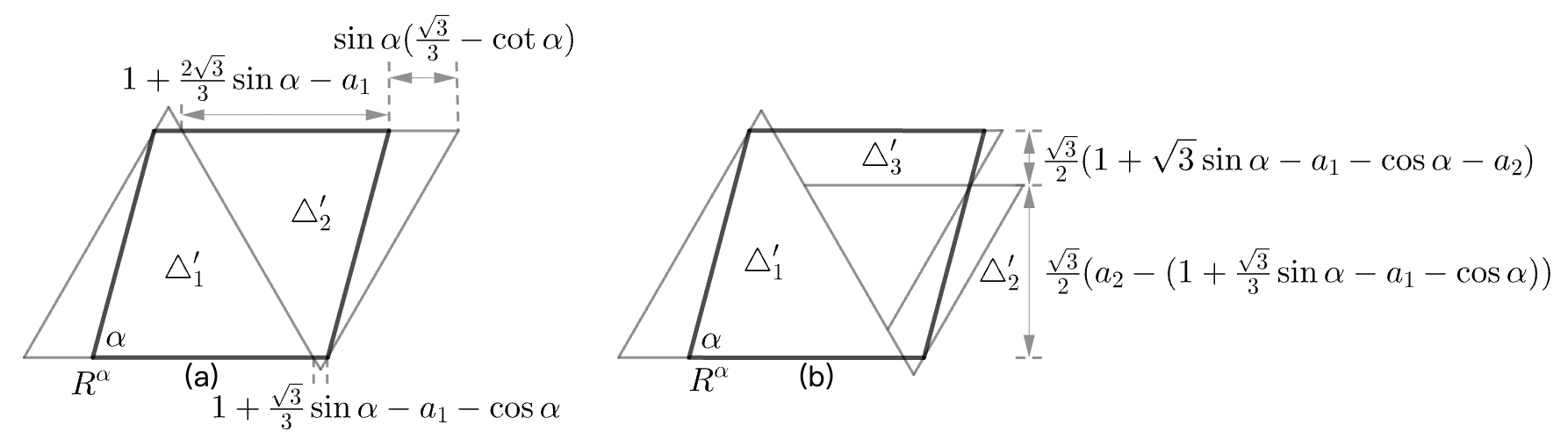}
	\caption{$\frac{2\sqrt{3}}{3}\sin\alpha \leq a_{1}<1 + \frac{\sqrt{3}}{3}\sin\alpha-\cos\alpha$}
	\label{Fig.8}
\end{figure}

We assume that $a_{2}<1+\sqrt{3}\sin\alpha-a_{1}-\cos\alpha$, otherwise $R^{\alpha}$ can be parallel covered by $\triangle_{1}$ and $\triangle_{2}$, as shown in Fig.~\ref{Fig.8}(a).

{\textit Subcase 2.1}: $\frac{3}{2}(1+\frac{\sqrt{3}}{3}\sin\alpha-a_{1}-\cos\alpha)<a_{2}< 1+\sqrt{3}\sin\alpha-a_{1}-\cos\alpha$.

We assume that
$a_{3}<1+\frac{2\sqrt{3}}{3}\sin
\alpha-a_{1}+\frac{\sqrt{3}}{2}(1+\sqrt{3}\sin\alpha-a_{1}-\cos\alpha-a_{2})(\frac{\sqrt{3}}{3}-\cot\alpha)$, otherwise $R^{\alpha}$ can be parallel covered by $\triangle_{1}$, $\triangle_{2}$ and $\triangle_{3}$, as shown in Fig.~\ref{Fig.8}(b).

(i) $1+\frac{2\sqrt{3}}{3}\sin\alpha-a_{1}\!\leq \!a_{3}\!<\! 1+\frac{2\sqrt{3}}{3}\sin\alpha-a_{1}+\frac{\sqrt{3}}{2}(1+\sqrt{3}\sin\alpha-a_{1}-\cos\alpha-a_{2})(\frac{\sqrt{3}}{3}-\cot\alpha)$.

\begin{figure}[htb]
	\centering
	\includegraphics[width=14cm]{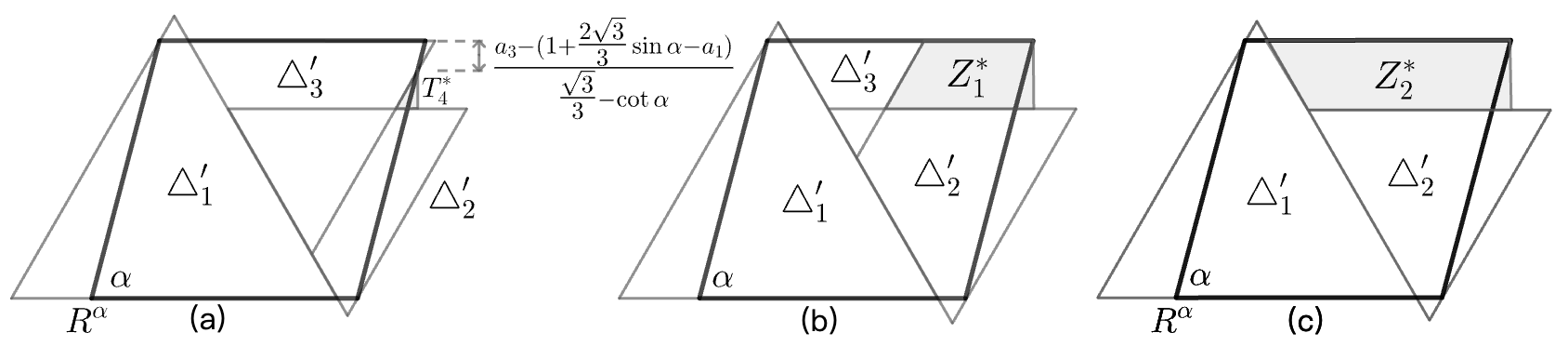}
	\caption{$a_{2}>\frac{3}{2}(1+\frac{\sqrt{3}}{3}\sin\alpha-a_{1}-\cos\alpha)$}
	\label{Fig.9}
\end{figure}

We place $\triangle_{1}$, $\triangle_{2}$ and $\triangle_{3}$ as in Fig.~\ref{Fig.9}(a).
The remaining equilateral triangles are used for the covering of the right triangle $T_{4}^{\ast}\supset R^{\alpha}\setminus (\triangle'_{1}\cup\triangle'_{2}\cup\triangle'_{3})$ with leg lengths
$\frac{\sqrt{3}}{3}\cdot(\frac{\sqrt{3}}{2}(1+\sqrt{3}\sin\alpha-a_{1}-\cos\alpha-a_{2})
-\frac{a_{3}-(1+\frac{2\sqrt{3}}{3}\sin\alpha-a_{1})}{\frac{\sqrt{3}}{3}-\cot\alpha})$ and $\frac{\sqrt{3}}{2}(1+\sqrt{3}\sin\alpha-a_{1}-\cos\alpha-a_{2})
-\frac{a_{3}-(1+\frac{2\sqrt{3}}{3}\sin\alpha-a_{1})}{\frac{\sqrt{3}}{3}-\cot\alpha}$.
By Corollary~\ref{cor:6} and $a_{4}\leq a_{3}$, we deduce that if $R^{\alpha}$ cannot be covered, then
\begin{align*}
	\sum A(\triangle_{n})
	&< \frac{\sqrt{3}}{4}(a_{1}^2+a_{2}^2+a_{3}^2)\\
	&~~~+\frac{\sqrt{3}}{6}\cdot\left(\frac{\sqrt{3}}{2}(1+\sqrt{3}\sin\alpha-a_{1}-\cos\alpha-a_{2})
	-\frac{a_{3}-(1+\frac{2\sqrt{3}}{3}\sin\alpha-a_{1})}{\frac{\sqrt{3}}{3}-\cot\alpha}\right)^2\\[-2pt]
	&~~~+\frac{3}{4}a_{3}\cdot\left(\frac{\sqrt{3}}{2}(1+\sqrt{3}\sin\alpha-a_{1}-\cos\alpha-a_{2})
	-\frac{a_{3}-(1+\frac{2\sqrt{3}}{3}\sin\alpha-a_{1})}{\frac{\sqrt{3}}{3}-\cot\alpha}\right).
\end{align*}

The maximum of the right-hand side is
$\frac{7\sqrt{3}}{12}\sin^{2}\alpha+\frac{1}{4}\sin\alpha\cos\alpha-\frac{5}{8}\sin\alpha-\frac{5\sqrt{3}}{8}\cos\alpha+\frac{7\sqrt{3}}{8}$
provided that the feasible region is given by $\frac{2\sqrt{3}}{3}\sin\alpha \leq a_{1}<1 + \frac{\sqrt{3}}{3}\sin\alpha-\cos\alpha$, $\frac{3}{2}(1+\frac{\sqrt{3}}{3}\sin\alpha-a_{1}-\cos\alpha)<a_{2}< 1+\sqrt{3}\sin\alpha-a_{1}-\cos\alpha$, $1+\frac{2\sqrt{3}}{3}\sin\alpha-a_{1}\leq a_{3}< 1+\frac{2\sqrt{3}}{3}\sin\alpha-a_{1}+\frac{\sqrt{3}}{2}(1+\sqrt{3}\sin\alpha-a_{1}-\cos\alpha-a_{2})(\frac{\sqrt{3}}{3}-\cot\alpha)$ and $a_{1}\geq a_{2}\geq a_{3}$.
This maximum, which is attained at $a_{1}=a_{2}=\frac{2\sqrt{3}}{3}\sin\alpha$ and
$a_{3}=1$, is less than $\frac{\sqrt{3}}{4}(1+\frac{2\sqrt{3}}{3}\sin\alpha)^2$.
This leads to a contradiction.

(ii) $\frac{3-\sqrt{6}}{2}\cdot(1+\frac{2\sqrt{3}}{3}\sin\alpha-a_{1})\leq a_{3}< 1+\frac{2\sqrt{3}}{3}\sin\alpha-a_{1}$.

We place $\triangle_{1}$, $\triangle_{2}$ and $\triangle_{3}$ as in Fig.~\ref{Fig.9}(b).
The remaining equilateral triangles are used for the covering of the trapezoid
$Z_{1}^{\ast}\supset R^{\alpha}\setminus(\triangle'_{1}\cup\triangle'_{2}\cup\triangle'_{3})$
 with base lengths
 $1+\frac{2\sqrt{3}}{3}\sin\alpha-a_{1}-a_{3}$ and
$\frac{3}{2}-\frac{3}{2}a_{1}-\frac{1}{2}a_{2}-a_{3}-\frac{1}{2}\cos\alpha+\frac{7\sqrt{3}}{6}\sin\alpha$ and height $\frac{\sqrt{3}}{2}(1+\sqrt{3}\sin\alpha-a_{1}-\cos\alpha-a_{2})$.
By Corollary~\ref{cor:5} and $a_{4}\leq a_{3}$, we deduce that if $R^{\alpha}$ cannot be covered, then
\begin{align*}
	\sum A(\triangle_{n})
	&< \frac{\sqrt{3}}{4}(a_{1}^2+a_{2}^2+a_{3}^2)+\frac{\sqrt{3}}{4}(1+\sqrt{3}\sin\alpha-a_{1}-\cos\alpha-a_{2})\\
	&~~~\cdot(\frac{5}{2}-\frac{5}{2}a_{1}-\frac{1}{2}a_{2}-2a_{3}-\frac{1}{2}\cos\alpha+\frac{11\sqrt{3}}{6}\sin\alpha)\\
	&~~~+\frac{\sqrt{3}}{8}a_{3}\cdot(7-7a_{1}-3a_{2}-4a_{3}-3\cos\alpha
	+\frac{17\sqrt{3}}{3}\sin\alpha).
\end{align*}

The maximum of the right-hand side over the closure of the feasible region is $\frac{1}{16}\cdot(2\sqrt{3}\cos^2\alpha-4\sin\alpha\cos\alpha-(9\sqrt{2}+3\sqrt{3})\cos\alpha+6\sqrt{3}\sin^2\alpha+3(1+\sqrt{6})\sin\alpha-9\sqrt{2}+16\sqrt{3})$
provided that the feasible region is given by
$\frac{2\sqrt{3}}{3}\sin\alpha \leq a_{1}<1 + \frac{\sqrt{3}}{3}\sin\alpha-\cos\alpha$, $\frac{3}{2}(1+\frac{\sqrt{3}}{3}\sin\alpha-a_{1}-\cos\alpha)<a_{2}< 1+\sqrt{3}\sin\alpha-a_{1}-\cos\alpha$,
$\frac{3-\sqrt{6}}{2}\cdot(1+\frac{2\sqrt{3}}{3}\sin\alpha-a_{1})\leq a_{3}< 1+\frac{2\sqrt{3}}{3}\sin\alpha-a_{1}$
and $a_{1}\geq a_{2} \geq a_{3}$.
This maximum, which is attained at
$a_{1}=\frac{2\sqrt{3}}{3}\sin\alpha$ and $a_{2}=a_{3}=\frac{3-\sqrt{6}}{2}$,
is less than $\frac{\sqrt{3}}{4}(1+\frac{2\sqrt{3}}{3}\sin\alpha)^2$.
This leads to a contradiction.

(iii) $0<a_{3}< \frac{3-\sqrt{6}}{2}\cdot(1+\frac{2\sqrt{3}}{3}\sin\alpha-a_{1})$.

We place $\triangle_{1}$ and $\triangle_{2}$ as in Fig.~\ref{Fig.9}(c).
The remaining equilateral triangles are used for the covering of the trapezoid
$Z_{2}^{\ast}\supset R^{\alpha}\setminus (\triangle'_{1}\cup\triangle'_{2})$ with base lengths
$\frac{1}{2}-\frac{1}{2}a_{1}+\frac{1}{2}a_{2}+\frac{1}{2}\cos\alpha+\frac{\sqrt{3}}{6}\sin\alpha$ and $1+\frac{2\sqrt{3}}{3}\sin\alpha-a_{1}$ and height $\frac{\sqrt{3}}{2}(1+\sqrt{3}\sin\alpha-a_{1}-\cos\alpha-a_{2})$.
By Corollary~\ref{cor:5}, we deduce that if $R^{\alpha}$ cannot be covered, then
\begin{align*}
	\sum A(\triangle_{n})
	&< \frac{\sqrt{3}}{4}(a_{1}^2+a_{2}^2)\\
	&~~~\!+\!\frac{\sqrt{3}}{4}(1+\sqrt{3}\sin\alpha-a_{1}-\cos\alpha-a_{2})(\frac{3}{2}-\frac{3}{2}a_{1}+\frac{1}{2}a_{2}+\frac{1}{2}\cos\alpha+\frac{5\sqrt{3}}{6}\sin\alpha)\\
	&~~~\!+\!\frac{\sqrt{3}}{2}a_{3}(\frac{5}{4}-\frac{5}{4}a_{1}-\frac{1}{4}a_{2}-\frac{1}{4}\cos\alpha+\frac{11\sqrt{3}}{12}\sin\alpha).
\end{align*}

The feasible region is defined by $\frac{2\sqrt{3}}{3}\sin\alpha \leq a_{1}<1 + \frac{\sqrt{3}}{3}\sin\alpha-\cos\alpha$, $\frac{3}{2}(1+\frac{\sqrt{3}}{3}\sin\alpha-a_{1}-\cos\alpha)<a_{2}<
1+\sqrt{3}\sin\alpha-a_{1}-\cos\alpha$,
$0<a_{3}< \frac{3-\sqrt{6}}{2}\cdot(1+\frac{2\sqrt{3}}{3}\sin\alpha-a_{1})$ and $a_{1}\geq a_{2} \geq a_{3}$.
The maximum of the right-hand side over the closure of the feasible region is attained at
$a_{1}=\frac{2\sqrt{3}}{3}\sin\alpha$ and $a_{2}=a_{3}=\frac{3-\sqrt{6}}{2}$.
Moreover, this maximum is less than $\frac{\sqrt{3}}{4}(1+\frac{2\sqrt{3}}{3}\sin\alpha)^2$.
This leads to a contradiction.

{\textit Subcase 2.2}: $0<a_{2}\leq \frac{3}{2}(1+\frac{\sqrt{3}}{3}\sin\alpha-a_{1}-\cos\alpha)$.
\begin{figure}[htb]
	\centering
	\includegraphics[width=7.5cm]{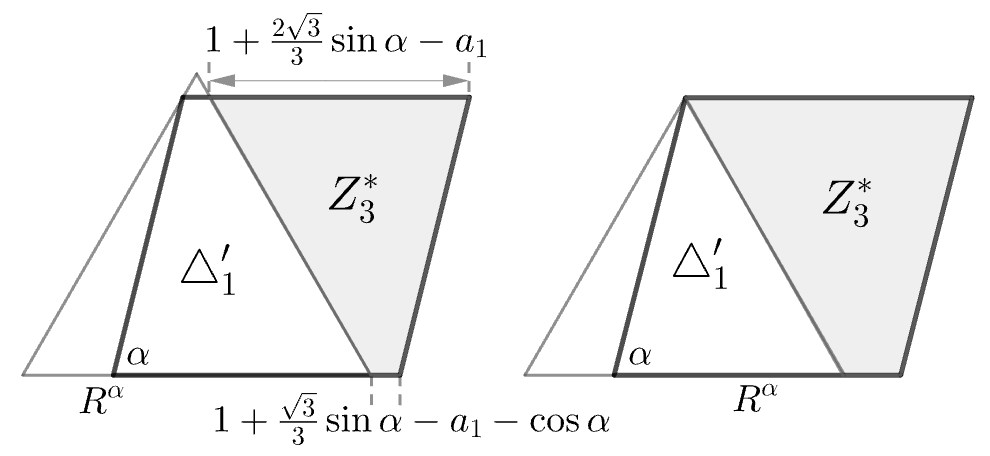}
	\caption{$0<a_{2}\leq \frac{3}{2}(1+\frac{\sqrt{3}}{3}\sin\alpha-a_{1}-\cos\alpha)$}
	\label{Fig.10}
\end{figure}

We place $\triangle_{1}$ as in Fig.~\ref{Fig.10}.
The remaining equilateral triangles are used for the covering of the trapezoid $Z^{\ast}_{3}=R^{\alpha}\setminus \triangle'_{1}$ with base lengths $1+\frac{\sqrt{3}}{3}\sin\alpha-a_{1}-\cos\alpha$ and $1+\frac{2\sqrt{3}}{3}\sin\alpha-a_{1}$ and height $\sin\alpha$.
By Lemma~\ref{lem:2}, we deduce that if $R^{\alpha}$ cannot be covered, then
\begin{align*}
	\sum A(\triangle_{n})
	&< \frac{\sqrt{3}}{4}a_{1}^{2}+\frac{1}{2}\sin\alpha(2+\sqrt{3}\sin\alpha-2a_{1}-\cos\alpha)\\
	&~~~+\frac{\sqrt{3}}{2}a_{2}(1+\frac{\sqrt{3}}{3}\sin\alpha-a_{1}-\cos\alpha)+\frac{\sqrt{3}}{4}a_{2}\sin\alpha(\sqrt{3}+\cot\alpha).
\end{align*}

The maximum is
$\frac{9\sqrt{3}\cos^{2}\alpha-12\sin\alpha\cos\alpha-27\sqrt{3}\cos\alpha+\sqrt{3}\sin^{2}\alpha+15\sin\alpha+18\sqrt{3}}{24}$
provided that the feasible region is given by
$\frac{2\sqrt{3}}{3}\sin\alpha \leq a_{1}<1 + \frac{\sqrt{3}}{3}\sin\alpha-\cos\alpha$,
$0<a_{2}\leq \frac{3}{2}(1+\frac{\sqrt{3}}{3}\sin\alpha-a_{1}-\cos\alpha)$ and $a_{1}\geq a_{2}$.
This maximum, which is attained at $a_{1}=\frac{2\sqrt{3}}{3}\sin\alpha$ and $a_{2}=\frac{3}{2}(1-\frac{\sqrt{3}}{3}\sin\alpha-\cos\alpha)$, is less than
$\frac{\sqrt{3}}{4}(1+\frac{2\sqrt{3}}{3}\sin\alpha)^2$.
This leads to a contradiction.

{\textit Case 3}: $0< a_{1}<\frac{2\sqrt{3}}{3}\sin\alpha$.

{\textit Subcase 3.1}: $a_{2}>\frac{1}{2}(1+(\sin\alpha-\frac{\sqrt{3}}{2}a_{1})(\frac{\sqrt{3}}{3}+\cot\alpha))$.

{\textit Subcase 3.1.1}: $a_{2}+\frac{\sqrt{3}}{2}a_{3}(\frac{\sqrt{3}}{3}+\cot\alpha)\geq1+(\sin\alpha-\frac{\sqrt{3}}{2}a_{1})(\frac{\sqrt{3}}{3}+\cot\alpha)$.
\begin{figure}[htb]
	\centering
	\includegraphics[width=10cm]{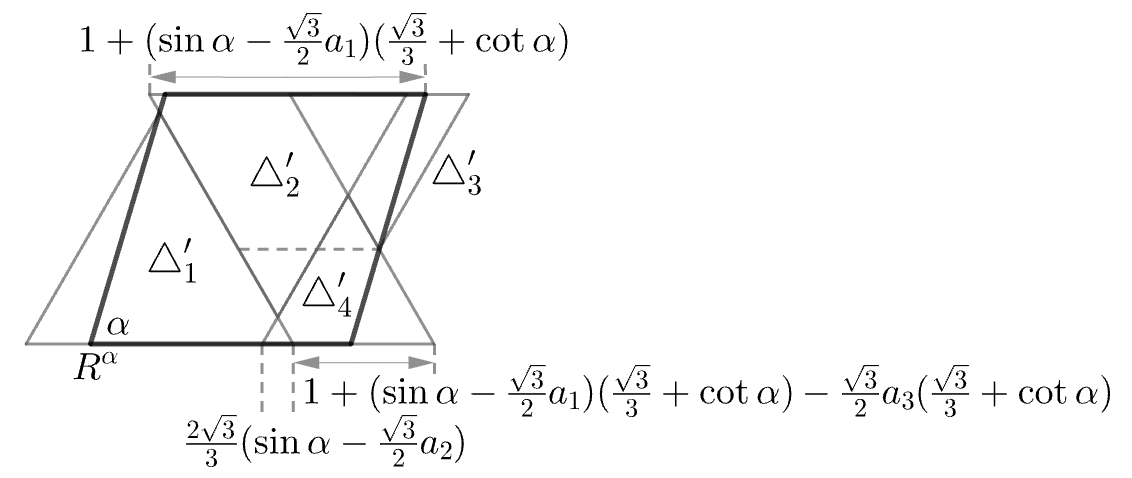}
	\caption{$a_{4}=1-a_{2}+\sqrt{3}\sin\alpha+\cos\alpha-\frac{1}{2}(a_{1}+a_{3})(1+\sqrt{3}\cot\alpha)$}
	\label{Fig.11}
\end{figure}

We assume that $a_{4}<1-a_{2}+\sqrt{3}\sin\alpha+\cos\alpha-\frac{1}{2}(a_{1}+a_{3})(1+\sqrt{3}\cot\alpha)$, otherwise $R^{\alpha}$ can be parallel covered by $\triangle_{1}$, $\triangle_{2}$, $\triangle_{3}$ and $\triangle_{4}$, as shown in Fig.~\ref{Fig.11}.

(i) $1-\frac{\sqrt{3}}{2}a_{1}(\frac{\sqrt{3}}{3}+\cot\alpha)<a_{4}<1-a_{2}+\sqrt{3}\sin\alpha+\cos\alpha-\frac{1}{2}(a_{1}+a_{3})(1+\sqrt{3}\cot\alpha)$.
\begin{figure}[h]
	\centering
	\includegraphics[width=10cm]{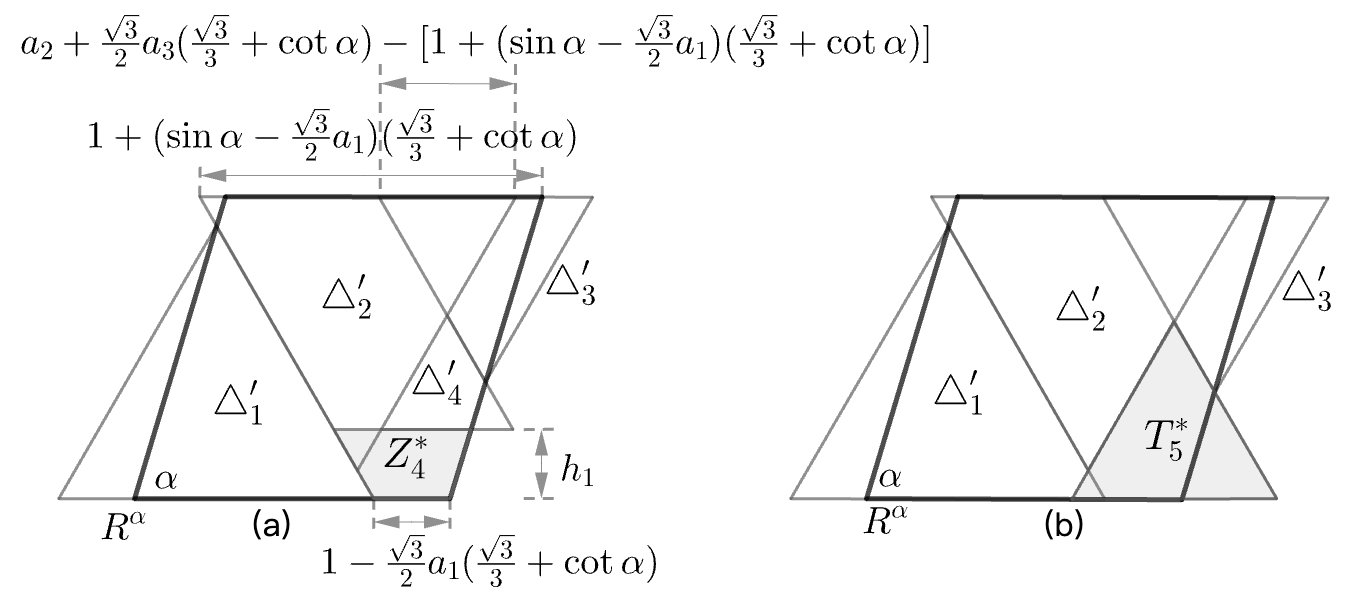}
	\caption{$0<a_{4}<1-a_{2}+\sqrt{3}\sin\alpha+\cos\alpha-\frac{1}{2}(a_{1}+a_{3})(1+\sqrt{3}\cot\alpha)$}
	\label{Fig.12}
\end{figure}

We place $\triangle_{1}$, $\triangle_{2}$, $\triangle_{3}$ and $\triangle_{4}$ as in Fig.~\ref{Fig.12}(a).
The remaining equilateral triangles are used for the covering of the trapezoid
$Z_{4}^{\ast}\supset R^{\alpha}\setminus (\triangle'_{1}\cup\triangle'_{2}\cup\triangle'_{3}\cup\triangle'_{4})$
with base lengths
$1-\frac{\sqrt{3}}{2}a_{1}(\frac{\sqrt{3}}{3}+\cot\alpha)$
and $1-\frac{\sqrt{3}}{2}a_{1}(\frac{\sqrt{3}}{3}+\cot\alpha)+h_{1}(\frac{\sqrt{3}}{3}+\cot\alpha)$
and height
$h_{1}:=\frac{\sqrt{3}}{2}(1-a_{2}-a_{4}+\sqrt{3}\sin\alpha+\cos\alpha-\frac{1}{2}(a_{1}+a_{3})(1+\sqrt{3}\cot\alpha))$.
By Lemma~\ref{lem:2} and $a_{5}\leq a_{4}$, we deduce that if $R^{\alpha}$ cannot be covered, then
\begin{align*}
	\sum A(\triangle_{n})
	&<\frac{\sqrt{3}}{4}(a_{1}^2+a_{2}^2+a_{3}^2+a_{4}^2)+\frac{1}{2}h_{1}\cdot(2-\sqrt{3}a_{1}(\frac{\sqrt{3}}{3}+\cot\alpha)+h_{1}(\frac{\sqrt{3}}{3}+\cot\alpha))\\
	&~~~+\frac{\sqrt{3}}{2}a_{4}(1-\frac{\sqrt{3}}{2}a_{1}(\frac{\sqrt{3}}{3}+\cot\alpha))
	+\frac{\sqrt{3}}{4}a_{4}h_{1}
	(\sqrt{3}+\cot\alpha).
\end{align*}

The maximum of the right-hand side over the closure of the feasible region is
$\sqrt{3}\sin^{2}\alpha+\frac{3\sqrt{3}}{4}(1-\frac{\sqrt{3}}{3}\sin\alpha-\cos\alpha)^2$
provided that the feasible region is given by
$0<a_{1}<\frac{2\sqrt{3}}{3}\sin\alpha$,
$a_{2}>\frac{1}{2} (1+(\sin\alpha-\frac{\sqrt{3}}{2}a_{1})(\frac{\sqrt{3}}{3}+\cot\alpha))$,
$a_{2}+\frac{\sqrt{3}}{2}a_{3}(\frac{\sqrt{3}}{3}+\cot\alpha)\geq1+(\sin\alpha-\frac{\sqrt{3}}{2}a_{1})(\frac{\sqrt{3}}{3}+\cot\alpha)$, $1-\frac{\sqrt{3}}{2}a_{1}(\frac{\sqrt{3}}{3}+\cot\alpha)<a_{4}<1-a_{2}+\sqrt{3}\sin\alpha+\cos\alpha-\frac{1}{2}(a_{1}+a_{3})(1+\sqrt{3}\cot\alpha)$ and $a_{1}\geq a_{2}\geq a_{3}\geq a_{4}$.
This maximum, which is attained at $a_{1}=a_{2}=a_{3}=\frac{2\sqrt{3}}{3}\sin\alpha$, $a_{4}=1-\frac{\sqrt{3}}{3}\sin\alpha-\cos\alpha$, is less than
$\frac{\sqrt{3}}{4}(1+\frac{2\sqrt{3}}{3}\sin\alpha)^2$.
This leads to a contradiction.

(ii) $0<a_{4}\leq 1-\frac{\sqrt{3}}{2}a_{1}(\frac{\sqrt{3}}{3}+\cot\alpha)$.

We place $\triangle_{1}$, $\triangle_{2}$ and $\triangle_{3}$ as in Fig.~\ref{Fig.12}(b).
The remaining equilateral triangles are used for the covering of the equilateral triangle $T_{5}^{\ast}\supset R^{\alpha}\setminus (\triangle'_{1}\cup\triangle'_{2}\cup\triangle'_{3})$ with side length
$1-a_{2}+\sqrt{3}\sin\alpha+\cos\alpha-\frac{1}{2}(a_{1}+a_{3})(1+\sqrt{3}\cot\alpha)$
and height $\frac{\sqrt{3}}{2}(1-a_{2}+\sqrt{3}\sin\alpha+\cos\alpha-\frac{1}{2}(a_{1}+a_{3})(1+\sqrt{3}\cot\alpha))$.
By Corollary~\ref{cor:4}, we deduce that if $R^{\alpha}$ cannot be covered, then
\begin{align*}
	\sum A(\triangle_{n})
	&< \frac{\sqrt{3}}{4}(a_{1}^2+a_{2}^2+a_{3}^2)\\
	&~~~+\frac{1}{2}\cdot\frac{2\sqrt{3}}{3}\cdot\frac{3}{4}
	(1-a_{2}+\sqrt{3}\sin\alpha+\cos\alpha-\frac{1}{2}(a_{1}+a_{3})(1+\sqrt{3}\cot\alpha))^2\\
	&~~~+\frac{\sqrt{3}}{4}a_{4}\cdot\frac{4\sqrt{3}}{3}\cdot\frac{\sqrt{3}}{2}(1-a_{2}+\sqrt{3}\sin\alpha+\cos\alpha-\frac{1}{2}(a_{1}+a_{3})(1+\sqrt{3}\cot\alpha)).
\end{align*}

The feasible region is defined by $0< a_{1}<\frac{2\sqrt{3}}{3}\sin\alpha$,
$a_{2}>\frac{1}{2} (1+(\sin\alpha-\frac{\sqrt{3}}{2}a_{1})(\frac{\sqrt{3}}{3}+\cot\alpha))$,
$a_{2}+\frac{\sqrt{3}}{2}a_{3}(\frac{\sqrt{3}}{3}+\cot\alpha)\geq1+(\sin\alpha-\frac{\sqrt{3}}{2}a_{1})(\frac{\sqrt{3}}{3}+\cot\alpha)$,
$0<a_{4}\leq 1-\frac{\sqrt{3}}{2}a_{1}(\frac{\sqrt{3}}{3}+\cot\alpha)$
and $a_{1}\geq a_{2}\geq a_{3}\geq a_{4}$.
The
maximum of the right-hand side is attained at
$a_{1}=a_{2}=a_{3}=\frac{\sqrt{3}+\sqrt{3}\cos\alpha+\sin\alpha}{2\sqrt{3}+3\cot\alpha}$
and
$a_{4}=\frac{3\sin(\alpha+\frac{\pi}{6})-\sin(2\alpha+\frac{\pi}{6})-1}
{3\cos\alpha+2\sqrt{3}\sin\alpha}$.
Moreover, this maximum is less than $\frac{\sqrt{3}}{4}(1+\frac{2\sqrt{3}}{3}\sin\alpha)^2$.
This leads to a contradiction.

{\textit Subcase 3.1.2}: $a_{2}+\frac{\sqrt{3}}{2}a_{3}(\frac{\sqrt{3}}{3}+\cot\alpha)<1+(\sin\alpha-\frac{\sqrt{3}}{2}a_{1})(\frac{\sqrt{3}}{3}+\cot\alpha)$.

{\textit Subcase 3.1.2.1}: $a_{3}>\frac{\sqrt{3}}{3}(1-(\frac{\sqrt{3}}{2}a_{1}+\frac{\sqrt{3}}{2}a_{2}-\sin\alpha)(\frac{\sqrt{3}}{3}+\cot\alpha))$.

\begin{figure}[htb]
	\centering
	\includegraphics[width=15cm]{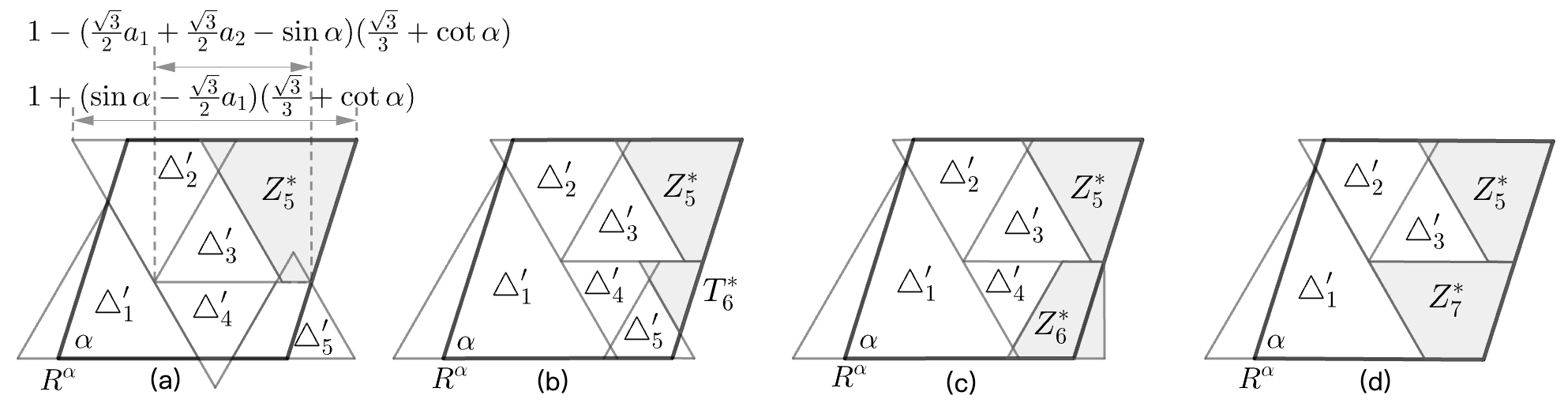}
	\caption{$a_{3}>\frac{\sqrt{3}}{3}(1-(\frac{\sqrt{3}}{2}a_{1}+\frac{\sqrt{3}}{2}a_{2}-\sin\alpha)(\frac{\sqrt{3}}{3}+\cot\alpha))$}
	\label{Fig.13}
\end{figure}

Throughout this subcase, let \(Z_{5}^{\ast}\) denote the
trapezoid with base lengths
$1-(\frac{\sqrt{3}}{2}a_{1}+\frac{\sqrt{3}}{2}a_{2}-\sin\alpha)(\frac{\sqrt{3}}{3}+\cot\alpha)-a_{3}$
and
$1+(\sin\alpha-\frac{\sqrt{3}}{2}a_{1})(\frac{\sqrt{3}}{3}+\cot\alpha)-a_{3}$
and height $\frac{\sqrt{3}}{2}a_{2}$, as shown in Fig.~\ref{Fig.13}.
In this subcase, we distinguish two possibilities according to \(a_{4}\).

{\textit Subcase 3.1.2.1.1}: $a_{4}> 1-(\frac{\sqrt{3}}{2}a_{1}+\frac{\sqrt{3}}{2}a_{2}-\sin\alpha)(\frac{\sqrt{3}}{3}+\cot\alpha)-a_{3}$.

In this subcase, we distinguish three possibilities according to \(a_{5}\).

(i) $a_{5}\geq 1-(\frac{\sqrt{3}}{2}a_{1}+\frac{\sqrt{3}}{2}a_{2}-\sin\alpha)(\frac{\sqrt{3}}{3}+\cot\alpha)-a_{4}+\frac{2\sqrt{3}}{3}(\sin\alpha-\frac{\sqrt{3}}{2}a_{2})$.

We place $\triangle_{1}$, $\triangle_{2}$, $\triangle_{3}$, $\triangle_{4}$ and $\triangle_{5}$ as in Fig.~\ref{Fig.13}(a).
The remaining equilateral triangles are used for the covering of the trapezoid $Z_{5}^{\ast}\supset R^{\alpha}\setminus (\triangle'_{1}\cup\triangle'_{2}\cup\triangle'_{3}\cup\triangle'_{4}\cup\triangle'_{5})$.
By Lemma~\ref{lem:2} and $a_{6}\leq a_{5}\leq a_{4}$, we deduce that if $R^{\alpha}$ cannot be covered, then
\begin{align*}
	\sum A(\triangle_{n})
	&<\frac{\sqrt{3}}{4}(a_{1}^2+a_{2}^2+a_{3}^2+2a_{4}^2)\\
	&~~~+\frac{\sqrt{3}}{4}a_{2}(2+(2\sin\alpha-\sqrt{3}a_{1}-\frac{\sqrt{3}}{2}a_{2})(\frac{\sqrt{3}}{3}+\cot\alpha)-2a_{3})\\
	&~~~+\frac{\sqrt{3}}{2}a_{4}(1\!-\!(\frac{\sqrt{3}}{2}a_{1}+\frac{\sqrt{3}}{2}a_{2}\!-\!\sin\alpha)(\frac{\sqrt{3}}{3}+\cot\alpha)-a_{3})+\frac{3}{8}a_{4}a_{2}(\sqrt{3}+\cot\alpha).
\end{align*}

The maximum over the closure of the feasible region is
$\frac{(7\sqrt{3}+3\cot\alpha)(1+\cos\alpha+\frac{\sin\alpha}{\sqrt{3}})^2}{4(2+\sqrt{3}\cot\alpha)^2}$ provided that the feasible region is given by
$0< a_{1}<\frac{2\sqrt{3}}{3}\sin\alpha$,
$a_{2}>\frac{1}{2} (1+(\sin\alpha-\frac{\sqrt{3}}{2}a_{1})(\frac{\sqrt{3}}{3}+\cot\alpha))$,
$a_{2}+\frac{\sqrt{3}}{2}a_{3}(\frac{\sqrt{3}}{3}+\cot\alpha)<1+(\sin\alpha-\frac{\sqrt{3}}{2}a_{1})(\frac{\sqrt{3}}{3}+\cot\alpha)$,
$a_{3}>\frac{\sqrt{3}}{3}(1-(\frac{\sqrt{3}}{2}a_{1}+\frac{\sqrt{3}}{2}a_{2}-\sin\alpha)(\frac{\sqrt{3}}{3}+\cot\alpha))$,
$a_{4}> 1-(\frac{\sqrt{3}}{2}a_{1}+\frac{\sqrt{3}}{2}a_{2}-\sin\alpha)(\frac{\sqrt{3}}{3}+\cot\alpha)-a_{3}$
and $a_{1}\geq a_{2}\geq a_{3}\geq a_{4}$.
This maximum, which is attained at $a_{1}=a_{2}=a_{3}=a_{4}=\frac{\sqrt{3}+\sqrt{3}\cos\alpha+\sin\alpha}
{2\sqrt{3}+3\cot\alpha}$, is less than
$\frac{\sqrt{3}}{4}(1+\frac{2\sqrt{3}}{3}\sin\alpha)^2$.
This leads to a contradiction.

(ii) $\sin\alpha-\frac{\sqrt{3}}{2}a_{2}\leq a_{5}<1-(\frac{\sqrt{3}}{2}a_{1}+\frac{\sqrt{3}}{2}a_{2}-\sin\alpha)(\frac{\sqrt{3}}{3}+\cot\alpha)-a_{4}+\frac{2\sqrt{3}}{3}(\sin\alpha-\frac{\sqrt{3}}{2}a_{2})$.

We place $\triangle_{1}$, $\triangle_{2}$, $\triangle_{3}$, $\triangle_{4}$ and $\triangle_{5}$ as in Fig.~\ref{Fig.13}(b).
The remaining equilateral triangles are used for the covering of $Z_{5}^{\ast}$ and $T^{\ast}_{6}$, where
$T_{6}^{\ast}$ has base length $1-(\frac{\sqrt{3}}{2}a_{1}+\frac{\sqrt{3}}{2}a_{2}-\sin\alpha)(\frac{\sqrt{3}}{3}+\cot\alpha)-a_{4}+\frac{2\sqrt{3}}{3}(\sin\alpha-\frac{\sqrt{3}}{2}(a_{2}+a_{5}))$ and height $\frac{1}{\frac{\sqrt{3}}{3}+\cot\alpha}(1-(\frac{\sqrt{3}}{2}a_{1}+\frac{\sqrt{3}}{2}a_{2}-\sin\alpha)(\frac{\sqrt{3}}{3}+\cot\alpha)-a_{4}+\frac{2\sqrt{3}}{3}
(\sin\alpha-\frac{\sqrt{3}}{2}(a_{2}+a_{5})))$.
By Corollaries~\ref{cor:3} and~\ref{cor:4}, together with $a_{i}\leq a_{5}$ for $i \geq 6$, we deduce that if $R^{\alpha}$ cannot be covered, then
\begin{align*}
	\sum A(\triangle_{n})
	&<\frac{\sqrt{3}}{4}(a_{1}^2+a_{2}^2+a_{3}^2+a_{4}^2+a_{5}^2)\\
	&~~~+\frac{\sqrt{3}}{4}a_{2}(2+(2\sin\alpha-\sqrt{3}a_{1}-\frac{\sqrt{3}}{2}a_{2})(\frac{\sqrt{3}}{3}+\cot\alpha)-2a_{3})\\
	&~~~+\frac{\sqrt{3}}{2}a_{5}(1-(\frac{\sqrt{3}}{2}a_{1}+\frac{\sqrt{3}}{2}a_{2}-\sin\alpha)(\frac{\sqrt{3}}{3}+\cot\alpha)-a_{3})\!+\!\frac{3}{8}a_{5}a_{2}(\sqrt{3}+\cot\alpha)\\
	 &~~~+\frac{(1-(\frac{\sqrt{3}}{2}a_{1}+\frac{\sqrt{3}}{2}a_{2}-\sin\alpha)(\frac{\sqrt{3}}{3}+\cot\alpha)-a_{4}+\frac{2\sqrt{3}}{3}(\sin\alpha-\frac{\sqrt{3}}{2}(a_{2}+a_{5})))^2}{2(\frac{\sqrt{3}}{3}+\cot\alpha)} \\
	 &~~~+\frac{\sqrt{3}}{4}a_{5}\frac{1\!-\!(\frac{\sqrt{3}}{2}a_{1}\!+\!\frac{\sqrt{3}}{2}a_{2}\!-\!\sin\alpha)(\frac{\sqrt{3}}{3}+\cot\alpha)\!-\!a_{4}+\frac{2\sqrt{3}}{3}(\sin\alpha\!-\!\frac{\sqrt{3}}{2}(a_{2}+a_{5}))}{\frac{\sqrt{3}}{3}+\cot\alpha}\\
	&~~~\cdot(\sqrt{3}+\cot\alpha)+\frac{\sqrt{3}}{4}a_{5}^2.
\end{align*}

The feasible region is defined by
$0< a_{1}<\frac{2\sqrt{3}}{3}\sin\alpha$,
$a_{2}>\frac{1}{2} (1+(\sin\alpha-\frac{\sqrt{3}}{2}a_{1})(\frac{\sqrt{3}}{3}+\cot\alpha))$,
$a_{2}+\frac{\sqrt{3}}{2}a_{3}(\frac{\sqrt{3}}{3}+\cot\alpha)<1+(\sin\alpha-\frac{\sqrt{3}}{2}a_{1})(\frac{\sqrt{3}}{3}+\cot\alpha)$, $a_{3}>\frac{\sqrt{3}}{3}(1-(\frac{\sqrt{3}}{2}a_{1}+\frac{\sqrt{3}}{2}a_{2}-\sin\alpha)(\frac{\sqrt{3}}{3}+\cot\alpha))$,
$a_{4}> 1-(\frac{\sqrt{3}}{2}a_{1}+\frac{\sqrt{3}}{2}a_{2}-\sin\alpha)(\frac{\sqrt{3}}{3}+\cot\alpha)-a_{3}$,
$\sin\alpha-\frac{\sqrt{3}}{2}a_{2}\leq a_{5}<1-(\frac{\sqrt{3}}{2}a_{1}+\frac{\sqrt{3}}{2}a_{2}-\sin\alpha)(\frac{\sqrt{3}}{3}+\cot\alpha)-a_{4}+\frac{2\sqrt{3}}{3}(\sin\alpha-\frac{\sqrt{3}}{2}a_{2})$
and $a_{1}\geq a_{2}\geq a_{3}\geq a_{4}\geq a_{5}$.
The
maximum of the right-hand side over the closure of the feasible region is attained at
$a_{1}=a_{2}=\tfrac{2(\sqrt{3}+\sin\alpha+\sqrt{3}\cos\alpha)}{5\sqrt{3}+3\cot\alpha}$
and
$a_{3}=a_{4}=a_{5}=\frac{3\sin 2\alpha - 7\sqrt{3}\cos 2\alpha- 6\cos(\alpha+\frac{\pi}{6})+ 4\sqrt{3}}{6\left(5\sin\alpha+\sqrt{3}\cos\alpha\right)}$.
Moreover, this maximum is less than $\frac{\sqrt{3}}{4}(1+\frac{2\sqrt{3}}{3}\sin\alpha)^2$.
This leads to a contradiction.

(iii) $0<a_{5}< \sin\alpha-\frac{\sqrt{3}}{2}a_{2}$.

We place $\triangle_{1}$, $\triangle_{2}$, $\triangle_{3}$ and $\triangle_{4}$ as in Fig.~\ref{Fig.13}(c).
The remaining equilateral triangles are used for the covering of $Z_{5}^{\ast}$ and $Z_{6}^{\ast}$, where
$Z_{6}^{\ast}$ has base lengths
$1-(\frac{\sqrt{3}}{2}a_{1}+\frac{\sqrt{3}}{2}a_{2}-\sin\alpha)(\frac{\sqrt{3}}{3}+\cot\alpha)-a_{4}$ and $1-(\frac{\sqrt{3}}{2}a_{1}+\frac{\sqrt{3}}{2}a_{2}-\sin\alpha)(\frac{\sqrt{3}}{3}+\cot\alpha)
-a_{4}+\frac{\sqrt{3}}{3}(\sin\alpha-\frac{\sqrt{3}}{2}a_{2})$
and height $\sin\alpha-\frac{\sqrt{3}}{2}a_{2}$.
By Corollaries~\ref{cor:3} and~\ref{cor:5}, together with $a_{i}\leq a_{5}$ for $i \geq 6$, we deduce that if $R^{\alpha}$ cannot be covered, then
\begin{align*}
	\sum A(\triangle_{n})
	&< \frac{\sqrt{3}}{4}(a_{1}^2+a_{2}^2+a_{3}^2+a_{4}^2)\\
	&~~~+\frac{\sqrt{3}}{4}a_{2}(2+(2\sin\alpha-\sqrt{3}a_{1}-\frac{\sqrt{3}}{2}a_{2})(\frac{\sqrt{3}}{3}+\cot\alpha)-2a_{3})\\
	&~~~+\frac{\sqrt{3}}{2}a_{5}(1-(\frac{\sqrt{3}}{2}a_{1}+\frac{\sqrt{3}}{2}a_{2}-\sin\alpha)(\frac{\sqrt{3}}{3}+\cot\alpha)-a_{3})\\
	&~~~+\frac{3}{8}a_{5}a_{2}(\sqrt{3}+\cot\alpha)+\frac{1}{2}(\sin\alpha-\frac{\sqrt{3}}{2}a_{2})\\
	&~~~\cdot(2-(a_{1}+a_{2})(1+\sqrt{3}\cot\alpha)+\sqrt{3}\sin\alpha	+2\cos\alpha-\frac{1}{2}a_{2}-2a_{4})\\
	&~~~+\frac{\sqrt{3}}{2}a_{5}(1-\frac{1}{2}(a_{1}+a_{2})(1+\sqrt{3}\cot\alpha)+\frac{5\sqrt{3}}{6}\sin\alpha+\cos\alpha-a_{4}-\frac{3}{4}a_{2})\\
	&~~~+\frac{\sqrt{3}}{4}a_{5}^2.
\end{align*}

The feasible region is defined by
$0<a_{1}<\frac{2\sqrt{3}}{3}\sin\alpha$,
$a_{2}>\frac{1}{2}(1+(\sin\alpha-\frac{\sqrt{3}}{2}a_{1})(\frac{\sqrt{3}}{3}+\cot\alpha))$,
$a_{2}+\frac{\sqrt{3}}{2}a_{3}(\frac{\sqrt{3}}{3}+\cot\alpha)<1+(\sin\alpha-\frac{\sqrt{3}}{2}a_{1})(\frac{\sqrt{3}}{3}+\cot\alpha)$, $a_{3}>\frac{\sqrt{3}}{3}(1-(\frac{\sqrt{3}}{2}a_{1}+\frac{\sqrt{3}}{2}a_{2}-\sin\alpha)(\frac{\sqrt{3}}{3}+\cot\alpha))$,
$a_{4}>1-(\frac{\sqrt{3}}{2}a_{1}+\frac{\sqrt{3}}{2}a_{2}-\sin\alpha)(\frac{\sqrt{3}}{3}+\cot\alpha)-a_{3}$,
$0<a_{5}<\sin\alpha-\frac{\sqrt{3}}{2}a_{2}$ and $a_{1}\geq a_{2}\geq a_{3}\geq a_{4}\geq a_{5}$.
The
maximum of the right-hand side over the closure of the feasible region is attained at
$a_{1}=\frac{2\sin\alpha}{\sqrt{3}}+\frac{2(2\sqrt{3}-\cot\alpha-4\sin\alpha)}
{(2\sqrt{3}-\cot\alpha)(1+\sqrt{3}\cot\alpha)}$,
$a_{2}=\frac{2\sin\alpha}{2\sqrt{3}-\cot\alpha}$ and
$a_{3}=a_{4}=a_{5}=\frac{\sin\alpha(\sqrt{3}-\cot\alpha)}{2\sqrt{3}-\cot\alpha}$.
Moreover, this maximum is less than $\frac{\sqrt{3}}{4}(1+\frac{2\sqrt{3}}{3}\sin\alpha)^2$.
This leads to a contradiction.

{\textit Subcase 3.1.2.1.2}: $0<a_{4}\leq 1-(\frac{\sqrt{3}}{2}a_{1}+\frac{\sqrt{3}}{2}a_{2}-\sin\alpha)(\frac{\sqrt{3}}{3}+\cot\alpha)-a_{3}$.

We place $\triangle_{1}$, $\triangle_{2}$ and $\triangle_{3}$ as in Fig.~\ref{Fig.13}(d).
The remaining equilateral triangles are used for the covering of $Z_{5}^{\ast}$ and $Z_{7}^{\ast}$, where
 $Z_{7}^{\ast}$ has base lengths $1-\frac{\sqrt{3}}{2}a_{1}(\frac{\sqrt{3}}{3}+\cot\alpha)$ and $1-(\frac{\sqrt{3}}{2}a_{1}+\frac{\sqrt{3}}{2}a_{2}-\sin\alpha)(\frac{\sqrt{3}}{3}+\cot\alpha)$ and height $\sin\alpha-\frac{\sqrt{3}}{2}a_{2}$.
By Corollary~\ref{cor:3} and $a_{i}\leq a_{4}$ for $i \geq 5$, we deduce that if $R^{\alpha}$ cannot be covered, then
\begin{align*}
	\sum A(\triangle_{n})
	&< \frac{\sqrt{3}}{4}(a_{1}^2+a_{2}^2+a_{3}^2)+\frac{\sqrt{3}}{4}a_{2}(2+(2\sin\alpha-\sqrt{3}a_{1}-\frac{\sqrt{3}}{2}a_{2})(\frac{\sqrt{3}}{3}+\cot\alpha)\!-\!2a_{3})\\
	&~~~+\frac{\sqrt{3}}{2}a_{4}(1-(\frac{\sqrt{3}}{2}a_{1}+\frac{\sqrt{3}}{2}a_{2}-\sin\alpha)(\frac{\sqrt{3}}{3}+\cot\alpha)-a_{3})\!+\!\frac{3}{8}a_{4}a_{2}(\sqrt{3}\!+\!\cot\alpha)\\
	&~~~+\frac{1}{2}(\sin\alpha-\frac{\sqrt{3}}{2}a_{2})(2-(\sqrt{3}a_{1}+\frac{\sqrt{3}}{2}a_{2}-\sin\alpha)(\frac{\sqrt{3}}{3}+\cot\alpha))\\ &~~~+\frac{\sqrt{3}}{2}a_{4}(1\!-\!\frac{\sqrt{3}}{2}a_{1}(\frac{\sqrt{3}}{3}+\cot\alpha))+\frac{\sqrt{3}}{4}a_{4}(\sin\alpha\!-\!\frac{\sqrt{3}}{2}a_{2})(\sqrt{3}+\cot\alpha)+\frac{\sqrt{3}}{4}a_{4}^2.
\end{align*}

The feasible region is defined by
$0<a_{1}<\frac{2\sqrt{3}}{3}\sin\alpha$,
$a_{2}>\frac{1}{2}(1+(\sin\alpha-\frac{\sqrt{3}}{2}a_{1})(\frac{\sqrt{3}}{3}+\cot\alpha))$,
$a_{2}+\frac{\sqrt{3}}{2}a_{3}(\frac{\sqrt{3}}{3}+\cot\alpha)<1+(\sin\alpha-\frac{\sqrt{3}}{2}a_{1})(\frac{\sqrt{3}}{3}+\cot\alpha)$, $a_{3}>\frac{\sqrt{3}}{3}(1-(\frac{\sqrt{3}}{2}a_{1}+\frac{\sqrt{3}}{2}a_{2}-\sin\alpha)(\frac{\sqrt{3}}{3}+\cot\alpha))$, $0<a_{4}\leq 1-(\frac{\sqrt{3}}{2}a_{1}+\frac{\sqrt{3}}{2}a_{2}-\sin\alpha)(\frac{\sqrt{3}}{3}+\cot\alpha)-a_{3}$ and $a_{1}\geq a_{2}\geq a_{3}\geq a_{4}$.
The
maximum of the right-hand side over the closure of the feasible region is attained at
$a_{1}=a_{2}=\tfrac{2(\sqrt{3}+\sin\alpha+\sqrt{3}\cos\alpha)}{5\sqrt{3}+3\cot\alpha}$, $a_{3}=\frac{(\sqrt{3}-\cot\alpha)(\sqrt{3}+\sin\alpha+\sqrt{3}\cos\alpha)}
{5\sqrt{3}+3\cot\alpha}$ and $a_{4}=\frac{(3-\sqrt{3})(3-\sqrt{3}\cot\alpha)(\sqrt{3}+\sin\alpha+\sqrt{3}\cos\alpha)}
{3(5\sqrt{3}+3\cot\alpha)}$.
Moreover, this maximum is less than
$\frac{\sqrt{3}}{4}(1+\frac{2\sqrt{3}}{3}\sin\alpha)^2$.
This leads to a contradiction.

{\textit Subcase 3.1.2.2}: $0<a_{3}\leq \frac{\sqrt{3}}{3}(1-(\frac{\sqrt{3}}{2}a_{1}+\frac{\sqrt{3}}{2}a_{2}-\sin\alpha)(\frac{\sqrt{3}}{3}+\cot\alpha))$.
\begin{figure}[htb]
	\centering
	\includegraphics[width=11cm]{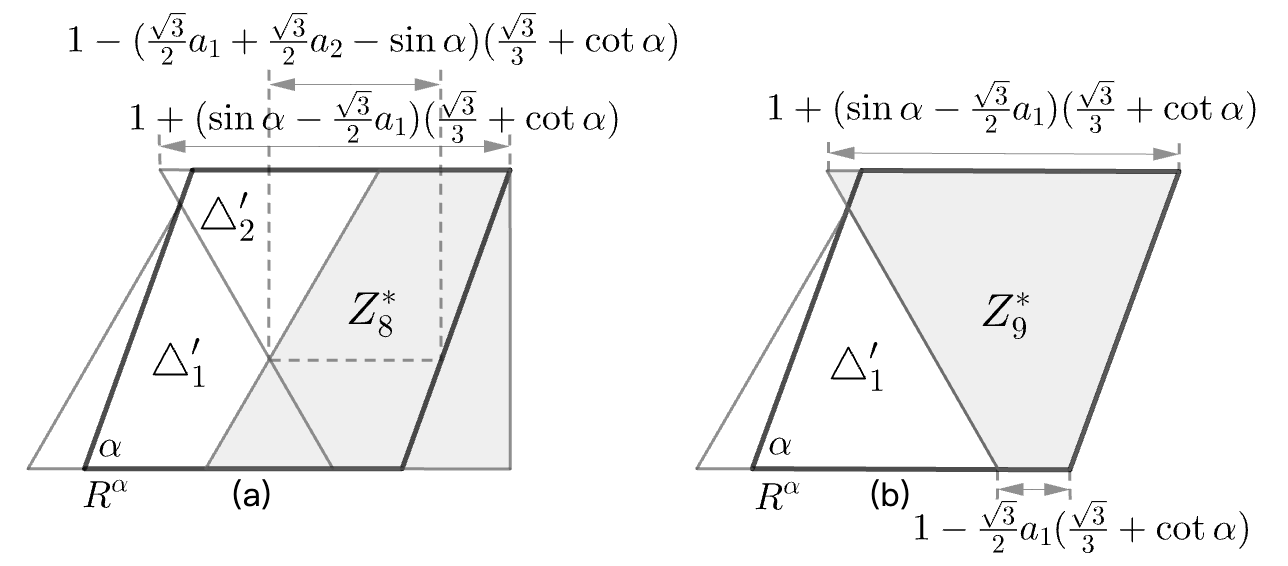}
	\caption{$0< a_{1}<\frac{2\sqrt{3}}{3}\sin\alpha$}
	\label{Fig.14}
\end{figure}

We place $\triangle_{1}$ and $\triangle_{2}$ as in Fig.~\ref{Fig.14}(a).
The remaining equilateral triangles are used for the covering of the trapezoid
$Z_{8}^{\ast}\supset R^{\alpha}\setminus (\triangle'_{1}\cup\triangle'_{2})$ with base lengths $1+(\sin\alpha-\frac{\sqrt{3}}{2}a_{1})(\frac{\sqrt{3}}{3}+\cot\alpha)-a_{2}$ and $1+(\sin\alpha-\frac{\sqrt{3}}{2}a_{1})(\frac{\sqrt{3}}{3}+\cot\alpha)-a_{2}+\frac{\sqrt{3}}{3}\sin\alpha$ and height $\sin\alpha$.
By Corollary~\ref{cor:5}, we deduce that if $R^{\alpha}$ cannot be covered, then
\begin{align*}
	\sum A(\triangle_{n})
	&< \frac{\sqrt{3}}{4}(a_{1}^2+a_{2}^2)+\frac{1}{2}\sin\alpha(2+2(\sin\alpha-\frac{\sqrt{3}}{2}a_{1})(\frac{\sqrt{3}}{3}+\cot\alpha)-2a_{2}+\frac{\sqrt{3}}{3}\sin\alpha)\\
	&~~~+\frac{\sqrt{3}}{2}a_{3}(1+(\sin\alpha-\frac{\sqrt{3}}{2}a_{1})(\frac{\sqrt{3}}{3}+\cot\alpha)-a_{2}+\frac{\sqrt{3}}{2}\sin\alpha).
\end{align*}

The feasible region is defined by
$0< a_{1}<\frac{2\sqrt{3}}{3}\sin\alpha$,
$a_{2}>\frac{1}{2}(1+(\sin\alpha-\frac{\sqrt{3}}{2}a_{1})(\frac{\sqrt{3}}{3}+\cot\alpha))$,
$a_{2}+\frac{\sqrt{3}}{2}a_{3}(\frac{\sqrt{3}}{3}+\cot\alpha)<1+(\sin\alpha-\frac{\sqrt{3}}{2}a_{1})(\frac{\sqrt{3}}{3}+\cot\alpha)$, $0<a_{3}\leq \frac{\sqrt{3}}{3}(1-(\frac{\sqrt{3}}{2}a_{1}+\frac{\sqrt{3}}{2}a_{2}-\sin\alpha)(\frac{\sqrt{3}}{3}+\cot\alpha))$ and $a_{1}\geq a_{2}\geq a_{3}$.
The
maximum of the right-hand side over the closure of the feasible region is attained at
$a_{1}=a_{2}=\tfrac{2(\sqrt{3}+\sin\alpha+\sqrt{3}\cos\alpha)}{5\sqrt{3}+3\cot\alpha}$ and $a_{3}=\frac{(\sqrt{3}-\cot\alpha)(\sqrt{3}+\sin\alpha+\sqrt{3}\cos\alpha)}{5\sqrt{3}+3\cot\alpha}$.
Moreover, this maximum is less than $\frac{\sqrt{3}}{4}(1+\frac{2\sqrt{3}}{3}\sin\alpha)^2$.
This leads to a contradiction.

{\textit Subcase 3.2}: $0<a_{2}\leq \frac{1}{2}(1+(\sin\alpha-\frac{\sqrt{3}}{2}a_{1})(\frac{\sqrt{3}}{3}+\cot\alpha))$.

We place $\triangle_{1}$ as in Fig.~\ref{Fig.14}(b).
The remaining equilateral triangles are used for the covering of the trapezoid $Z_{9}^{\ast}\supset R^{\alpha}\setminus \triangle'_{1}$ with base lengths $ 1-\frac{\sqrt{3}}{2}a_{1}(\frac{\sqrt{3}}{3}+\cot\alpha)$ and $1+(\sin\alpha-\frac{\sqrt{3}}{2}a_{1})(\frac{\sqrt{3}}{3}+\cot\alpha)$ and height $\sin\alpha$.
By Lemma~\ref{lem:2}, we deduce that if $R^{\alpha}$ cannot be covered, then
\begin{align*}
	\sum A(\triangle_{n})
	&< \frac{\sqrt{3}}{4} a_{1}^{2}+\frac{1}{2}\sin\alpha(2+(\sin\alpha-\sqrt{3}a_{1})(\frac{\sqrt{3}}{3}+\cot\alpha))\\
	&~~~+\frac{\sqrt{3}}{2}a_{2}(1-\frac{\sqrt{3}}{2}a_{1}(\frac{\sqrt{3}}{3}+\cot\alpha))+\frac{\sqrt{3}}{4}a_{2}\sin\alpha(\sqrt{3}+\cot\alpha).
\end{align*}

The feasible region is defined by
$0< a_{1}<\frac{2\sqrt{3}}{3}\sin\alpha$,
$0<a_{2}\leq\frac{1}{2} (1+(\sin\alpha-\frac{\sqrt{3}}{2}a_{1})(\frac{\sqrt{3}}{3}+\cot\alpha))$ and
$a_{1}\geq a_{2}$.
The
maximum of the right-hand side is attained at
$a_{1}=a_{2}=\tfrac{2(\sqrt{3}+\sin\alpha+\sqrt{3}\cos\alpha)}{5\sqrt{3}+3\cot\alpha}$.
Moreover, this maximum is less than $\frac{\sqrt{3}}{4}(1+\frac{2\sqrt{3}}{3}\sin\alpha)^2$.
This leads to a contradiction.

By the discussions above we know that
$\varrho({R}^{\alpha}, \triangle)\leq \frac{\sqrt{3}(\sqrt{3}+2\sin\alpha)^{2}}{12\sin\alpha}$.
However, any equilateral triangle with side length less than
$1+\frac{2\sqrt{3}}{3}\sin\alpha$ cannot parallel cover ${R}^{\alpha}$, hence
$\varrho({R}^{\alpha}, \triangle)\geq \frac{\sqrt{3}(\sqrt{3}+2\sin\alpha)^{2}}{12\sin\alpha}$.
As a consequence, $\varrho({R}^{\alpha}, \triangle)=\frac{\sqrt{3}(\sqrt{3}+2\sin\alpha)^{2}}{12\sin\alpha}$.
\end{proof}

The angle $\frac{5\pi}{12}$ comes from Subcase~2.1(ii) of Theorem~\ref{thm:2}.
For the corresponding upper-bound expression and feasible region,
the maximum is attained at
$a_1=\frac{2\sqrt3}{3}\sin\alpha$, $a_2=a_3=\frac{3-\sqrt6}{2}$,
when $\alpha\in(\frac{\pi}{3},\frac{5\pi}{12}]$,
whereas it is attained at
$a_1=\frac{2\sqrt3}{3}\sin\alpha$,
$a_2=a_3=\frac32(1-\frac{\sqrt3}{3}\sin\alpha-\cos\alpha)$,
when
$\alpha\in(\frac{5\pi}{12},\frac{\pi}{2}]$.
At $\alpha=\frac{5\pi}{12}$, we have
$\frac{3}{2}(1-\frac{\sqrt3}{3}\sin\alpha-\cos\alpha)=\frac{3-\sqrt6}{2}$.
Hence this is the transition value at which
the relevant active boundary constraint changes.

\begin{thm}\label{thm:3}
 If $\frac{5\pi}{12}< \alpha \leq \frac{\pi}{2}$ and
 $\sum A(\triangle_{n}) \geq \frac{\sqrt{3}}{4}(1+\frac{2\sqrt{3}}{3}\sin\alpha)^{2}$,
 then $\{\triangle_{n}\}$ permits a parallel covering of $R^{\alpha}$ and thus
 $\varrho({R}^{\alpha}, \triangle)=\frac{\sqrt{3}(\sqrt{3}+2\sin\alpha)^{2}}{12\sin\alpha}$.
 \end{thm}

\begin{proof}
We claim that $\{\triangle_{n}\}$ permits a parallel covering of $R^{\alpha}$.
We may assume that $a_{1} < 1 + \frac{2\sqrt{3}}{3}\sin\alpha$, otherwise $R^{\alpha}$ can be parallel covered by $\triangle_{1}$.

{\textit Case 1}: $1 + \frac{\sqrt{3}}{3}\sin\alpha-\cos\alpha\leq a_{1}<1 + \frac{2\sqrt{3}}{3}\sin\alpha$.

The argument in this case is the same as that in Case~1 of Theorem~\ref{thm:2}.

{\textit Case 2}: $\frac{2\sqrt{3}}{3}\sin\alpha \leq a_{1}<1 + \frac{\sqrt{3}}{3}\sin\alpha-\cos\alpha$.
\begin{figure}[htb]
	\centering
	\includegraphics[width=12.3cm]{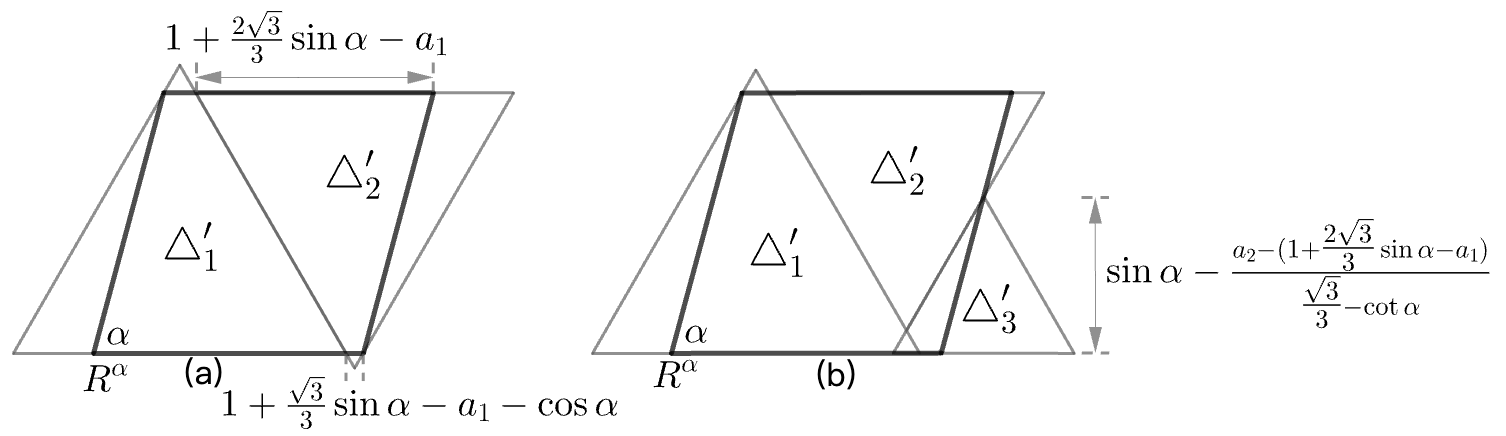}
	\caption{$a_{2}>1+\frac{2\sqrt{3}}{3}\sin\alpha-a_{1}$}
	\label{Fig.15}
\end{figure}

We assume that $a_{2}<1+\sqrt{3}\sin\alpha-a_{1}-\cos\alpha$, otherwise $R^{\alpha}$ can be parallel covered by $\triangle_{1}$ and $\triangle_{2}$, as shown in Fig.~\ref{Fig.15}(a).

{\textit Subcase 2.1}: $1+\frac{2\sqrt{3}}{3}\sin\alpha-a_{1}<a_{2}<1+\sqrt{3}\sin\alpha-a_{1}-\cos\alpha$.

We assume that $\frac{\sqrt{3}}{2}a_{3}<\sin\alpha-\frac{a_{2}-(1+\frac{2\sqrt{3}}{3}\sin\alpha-a_{1})}{\frac{\sqrt{3}}{3}-\cot\alpha}$, otherwise $R^{\alpha}$ can be parallel covered by $\triangle_{1}$, $\triangle_{2}$ and $\triangle_{3}$, as shown in Fig.~\ref{Fig.15}(b).

{\textit Subcase 2.1.1}: $\frac{4}{5}\cdot(\sin\alpha-\frac{a_{2}-(1+\frac{2\sqrt{3}}{3}\sin\alpha-a_{1})}{\frac{\sqrt{3}}{3}-\cot\alpha})\leq\frac{\sqrt{3}}{2}a_{3}<\sin\alpha-\frac{a_{2}-(1+\frac{2\sqrt{3}}{3}\sin\alpha-a_{1})}{\frac{\sqrt{3}}{3}-\cot\alpha}$.
\begin{figure}[htb]
	\centering
	\includegraphics[width=13.3cm]{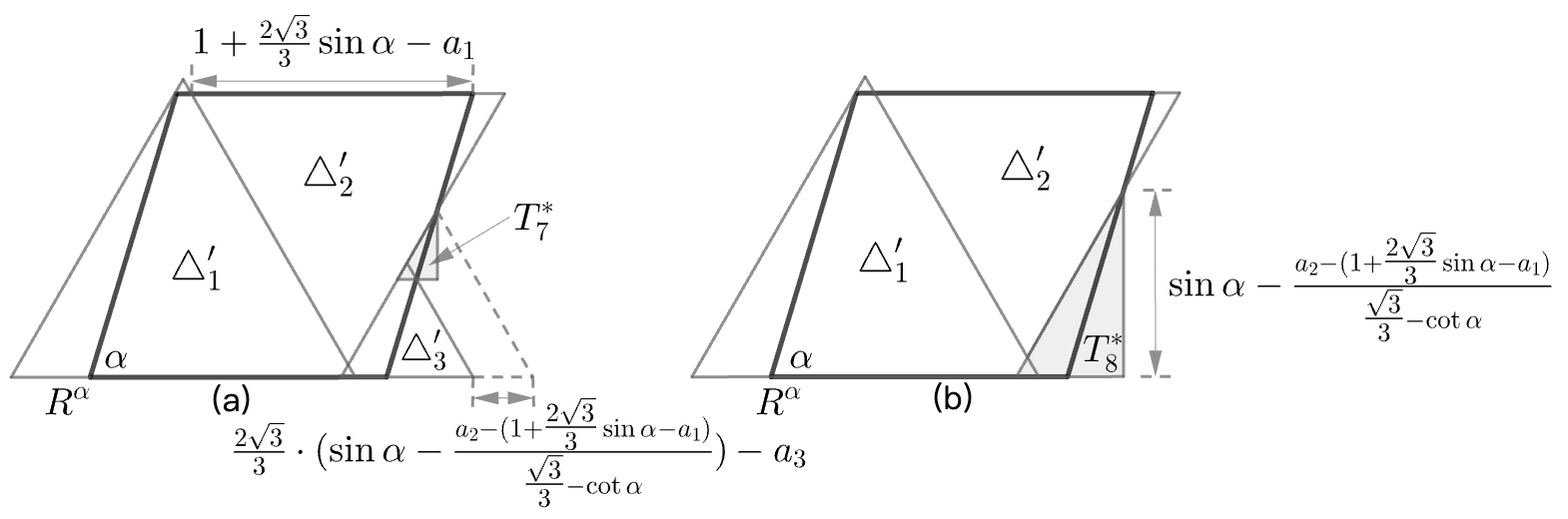}
	\caption{$0<\frac{\sqrt{3}}{2}a_{3}<\sin\alpha-\frac{a_{2}-(1+\frac{2\sqrt{3}}{3}\sin\alpha-a_{1})}{\frac{\sqrt{3}}{3}-\cot\alpha}$}
	\label{Fig.16}
\end{figure}

We place $\triangle_{1}$, $\triangle_{2}$ and $\triangle_{3}$ as in Fig.~\ref{Fig.16}(a).
The remaining equilateral triangles are used for the covering of the right triangle $T_{7}^{\ast}$ with leg lengths
$\frac{\sqrt{3}}{3}\cdot\frac{\frac{2\sqrt{3}}{3}\cdot(\sin\alpha-\frac{a_{2}-(1+\frac{2\sqrt{3}}{3}\sin\alpha-a_{1})}{\frac{\sqrt{3}}{3}-\cot\alpha})-a_{3}}{\frac{\sqrt{3}}{3}+\cot\alpha}$ and $\frac{\frac{2\sqrt{3}}{3}\cdot(\sin\alpha-\frac{a_{2}-(1+\frac{2\sqrt{3}}{3}\sin\alpha-a_{1})}{\frac{\sqrt{3}}{3}-\cot\alpha})-a_{3}}{\frac{\sqrt{3}}{3}+\cot\alpha}$.
By Corollary~\ref{cor:6} and $a_{4}\leq a_{3}$, we deduce that if $R^{\alpha}$ cannot be covered, then
\begin{align*}
	\sum A(\triangle_{n})
	&< \frac{\sqrt{3}}{4}(a_{1}^2+a_{2}^2+a_{3}^2)
	 +\frac{\sqrt{3}}{6}\cdot\left(\frac{\frac{2\sqrt{3}}{3}\cdot(\sin\alpha-\frac{a_{2}-(1+\frac{2\sqrt{3}}{3}\sin\alpha-a_{1})}{\frac{\sqrt{3}}{3}-\cot\alpha})-a_{3}}{\frac{\sqrt{3}}{3}+\cot\alpha}\right)^2\\
	&~~~+\frac{3}{4}a_{3}\cdot\frac{\frac{2\sqrt{3}}{3}\cdot(\sin\alpha-\frac{a_{2}-(1+\frac{2\sqrt{3}}{3}\sin\alpha-a_{1})}{\frac{\sqrt{3}}{3}-\cot\alpha})-a_{3}}{\frac{\sqrt{3}}{3}+\cot\alpha}.
\end{align*}

The maximum of the right-hand side over the closure of the feasible region is
$\frac{\sqrt{3}}{4}+\frac{1}{5}\sin\alpha+\frac{11\sqrt{3}}{25}\sin^2\alpha+\frac{2\sin^2\alpha(18\cot\alpha+7\sqrt{3})}{225(\frac{\sqrt{3}}{3}+\cot\alpha)^2}$ provided that the feasible region is given by $\frac{2\sqrt{3}}{3}\sin\alpha \leq a_{1}<1 + \frac{\sqrt{3}}{3}\sin\alpha-\cos\alpha$, $1+\frac{2\sqrt{3}}{3}\sin\alpha-a_{1}<a_{2}<1+\sqrt{3}\sin\alpha-a_{1}-\cos\alpha$, $\frac{4}{5}\cdot(\sin\alpha-\frac{a_{2}-(1+\frac{2\sqrt{3}}{3}\sin\alpha-a_{1})}{\frac{\sqrt{3}}{3}-\cot\alpha})\leq\frac{\sqrt{3}}{2}a_{3}<\sin\alpha-\frac{a_{2}-(1+\frac{2\sqrt{3}}{3}\sin\alpha-a_{1})}{\frac{\sqrt{3}}{3}-\cot\alpha}$ and $a_{1}\geq a_{2}\geq a_{3}$.
This maximum, which is attained at
$a_{1}=1+\frac{2\sqrt{3}}{15}\sin\alpha$ and $a_{2}=a_{3}=\frac{8\sqrt{3}}{15}\sin\alpha$, is less than $\frac{\sqrt{3}}{4}(1+\frac{2\sqrt{3}}{3}\sin\alpha)^2$.
This leads to a contradiction.

{\textit Subcase 2.1.2}: $0<\frac{\sqrt{3}}{2}a_{3}<\frac{4}{5}\cdot(\sin\alpha-\frac{a_{2}-(1+\frac{2\sqrt{3}}{3}\sin\alpha-a_{1})}{\frac{\sqrt{3}}{3}-\cot\alpha})$.

We place $\triangle_{1}$ and $\triangle_{2}$ as in Fig.~\ref{Fig.16}(b).
The remaining equilateral triangles are used for the covering of the right triangle $T_{8}^{\ast}\supset R^{\alpha}\setminus (\triangle'_{1}\cup\triangle'_{2})$ with leg lengths
$\frac{\sqrt{3}}{3}\cdot(\sin\alpha-\frac{a_{2}-(1+\frac{2\sqrt{3}}{3}\sin\alpha-a_{1})}{\frac{\sqrt{3}}{3}-\cot\alpha})$ and
$\sin\alpha-\frac{a_{2}-(1+\frac{2\sqrt{3}}{3}\sin\alpha-a_{1})}{\frac{\sqrt{3}}{3}-\cot\alpha}$.
By Corollary~\ref{cor:6}, we deduce that if $R^{\alpha}$ cannot be covered, then
\begin{align*}
	\sum A(\triangle_{n})
	&< \frac{\sqrt{3}}{4}(a_{1}^2+a_{2}^2)
	+\frac{\sqrt{3}}{6}\cdot\left(\sin\alpha-\frac{a_{2}-(1+\frac{2\sqrt{3}}{3}\sin\alpha-a_{1})}{\frac{\sqrt{3}}{3}-\cot\alpha}\right)^2\\
	&~~~+\frac{3}{4}a_{3}\cdot\left(\sin\alpha-\frac{a_{2}-(1+\frac{2\sqrt{3}}{3}\sin\alpha-a_{1})}{\frac{\sqrt{3}}{3}-\cot\alpha}\right).
\end{align*}

The maximum of the right-hand side over the closure of the feasible region is $\frac{\sqrt3}{4}+\frac15\sin\alpha+\frac{119\sqrt3}{150}\sin^2\alpha$ provided that the feasible region is given by $\frac{2\sqrt{3}}{3}\sin\alpha \leq a_{1}<1 + \frac{\sqrt{3}}{3}\sin\alpha-\cos\alpha$, $1+\frac{2\sqrt{3}}{3}\sin\alpha-a_{1}<a_{2}<1+\sqrt{3}\sin\alpha-a_{1}-\cos\alpha$, $0<\frac{\sqrt{3}}{2}a_{3}
<\frac{4}{5}\cdot(\sin\alpha-\frac{a_{2}-(1+\frac{2\sqrt{3}}{3}\sin\alpha-a_{1})}{\frac{\sqrt{3}}{3}-\cot\alpha})$ and
$a_{1}\geq a_{2}\geq a_{3}$.
This maximum, which is attained at
$a_{1}=1+\frac{2\sqrt{3}}{15}\sin\alpha$ and $a_{2}=a_{3}=\frac{8\sqrt{3}}{15}\sin\alpha$, is less than $\frac{\sqrt{3}}{4}(1+\frac{2\sqrt{3}}{3}\sin\alpha)^2$.
This leads to a contradiction.

{\textit Subcase 2.2}: $\frac{3-\sqrt{3}}{2}\cdot(1+\frac{2\sqrt{3}}{3}\sin\alpha-a_{1})<a_{2}\leq 1+\frac{2\sqrt{3}}{3}\sin\alpha-a_{1}$.

{\textit Subcase 2.2.1}: $a_{3}> 1+\frac{\sqrt{3}}{3}\sin\alpha-a_{1}-\cos\alpha$.
\begin{figure}[htb]
	\centering
\includegraphics[width=15cm]{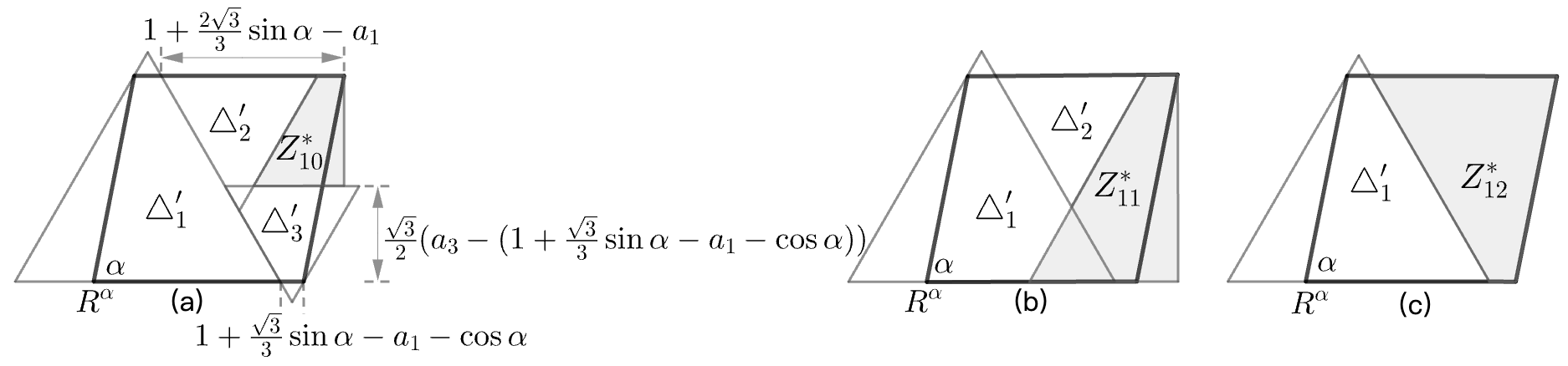}
	\caption{$0<a_{2}\leq 1+\frac{2\sqrt{3}}{3}\sin\alpha-a_{1}$}
	\label{Fig.17}
\end{figure}

We place $\triangle_{1}$, $\triangle_{2}$ and $\triangle_{3}$ as in Fig.~\ref{Fig.17}(a).
The remaining equilateral triangles are used for the covering of the trapezoid $Z^{\ast}_{10}\supset R^{\alpha}\setminus (\triangle'_{1}\cup\triangle'_{2}\cup\triangle'_{3})$ with base lengths $1+\frac{2\sqrt{3}}{3}\sin\alpha-a_{1}-a_{2}$ and
$\frac{3}{2}-\frac{3}{2}a_{1}-a_{2}-\frac{1}{2}a_{3}-\frac{1}{2}\cos\alpha+\frac{7\sqrt{3}}{6}\sin\alpha$ and height $\frac{\sqrt{3}}{2}(1+\sqrt{3}\sin\alpha-a_{1}-\cos\alpha-a_{3})$.
By Corollary~\ref{cor:5} and $a_{4}\leq a_{3}$, we deduce that if $R^{\alpha}$ cannot be covered, then
\begin{align*}
	\sum A(\triangle_{n})
	&<\frac{\sqrt{3}}{4}(a_{1}^2+a_{2}^2+a_{3}^2)+\frac{\sqrt{3}}{4}(1+\sqrt{3}\sin\alpha-a_{1}-\cos\alpha-a_{3})\\
	&~~~\cdot(\frac{5}{2}-\frac{5}{2}a_{1}-2a_{2}-\frac{1}{2}a_{3}-\frac{1}{2}\cos\alpha+\frac{11\sqrt{3}}{6}\sin\alpha)\\
	&~~~+\frac{\sqrt{3}}{8}a_{3}(7-7a_{1}-4a_{2}-3a_{3}-3\cos\alpha+\frac{17\sqrt{3}}{3}\sin\alpha).
\end{align*}

The maximum of the right-hand side over the closure of the feasible region is
$\frac{1}{16}\cdot(4\sqrt{3}\sin^{2}\alpha-4\sin\alpha\cos\alpha+3(1+\sqrt{3})\sin\alpha-3(3+\sqrt{3})\cos\alpha+15\sqrt{3}-9)$ provided that the feasible region is given by $\frac{2\sqrt{3}}{3}\sin\alpha \leq a_{1}<1 + \frac{\sqrt{3}}{3}\sin\alpha-\cos\alpha$, $\frac{3-\sqrt{3}}{2}\cdot(1+\frac{2\sqrt{3}}{3}\sin\alpha-a_{1})<a_{2}\leq 1+\frac{2\sqrt{3}}{3}\sin\alpha-a_{1}$, $a_{3}> 1+\frac{\sqrt{3}}{3}\sin\alpha-a_{1}-\cos\alpha$ and $a_{1}\geq a_{2} \geq a_{3}$.
This maximum, which is attained at
$a_{1}=\frac{2\sqrt{3}}{3}\sin\alpha$ and $a_{2}=a_{3}=\frac{3-\sqrt{3}}{2}$, is less than $\frac{\sqrt{3}}{4}(1+\frac{2\sqrt{3}}{3}\sin\alpha)^2$.
This leads to a contradiction.

{\textit Subcase 2.2.2}: $0<a_{3}\leq 1+\frac{\sqrt{3}}{3}\sin\alpha-a_{1}-\cos\alpha$.

We place $\triangle_{1}$ and $\triangle_{2}$ as in Fig.~\ref{Fig.17}(b).
The remaining equilateral triangles are used for the covering of the trapezoid $Z^{\ast}_{11}\supset R^{\alpha}\setminus (\triangle'_{1}\cup\triangle'_{2})$ with base lengths $1+\frac{2\sqrt{3}}{3}\sin\alpha-a_{1}-a_{2}$ and $1+\sqrt{3}\sin\alpha-a_{1}-a_{2}$ and height $\sin\alpha$.
By Corollary~\ref{cor:5}, we deduce that if $R^{\alpha}$ cannot be covered, then
\begin{align*}
	\sum A(\triangle_{n})
	&< \frac{\sqrt{3}}{4}(a_{1}^{2}+a^{2}_{2})+\frac{1}{2}\sin\alpha(2+\frac{5\sqrt{3}}{3}\sin\alpha-2a_{1}-2a_{2})\\
	&~~~+\frac{\sqrt{3}}{2}a_{3}(1+\frac{7\sqrt{3}}{6}\sin\alpha-a_{1}-a_{2}).
\end{align*}

The maximum of the right-hand side over the closure of the feasible region is
$\frac{\sqrt{3}}{4}\sin^{2}\alpha
+\frac{2+\sqrt{3}}{4}\sin\alpha
+\frac{\sqrt{3}}{2}-\frac{3}{4}\sin\alpha\cos\alpha
-\frac{3-\sqrt{3}}{4}\cos\alpha
-\frac{3}{8}$ provided that the feasible region is given by $\frac{2\sqrt{3}}{3}\sin\alpha \leq a_{1}<1 + \frac{\sqrt{3}}{3}\sin\alpha-\cos\alpha$, $\frac{3-\sqrt{3}}{2}\cdot(1+\frac{2\sqrt{3}}{3}\sin\alpha-a_{1})<a_{2}\leq 1+\frac{2\sqrt{3}}{3}\sin\alpha-a_{1}$, $0<a_{3}\leq 1+\frac{\sqrt{3}}{3}\sin\alpha-a_{1}-\cos\alpha$ and $a_{1}\geq a_{2} \geq a_{3}$.
This maximum, which is attained at
$a_{1}=\frac{2\sqrt{3}}{3}\sin\alpha$, $a_{2}=\frac{3-\sqrt{3}}{2}$ and $a_{3}=1-\frac{\sqrt{3}}{3}\sin\alpha-\cos\alpha$, is less than
$\frac{\sqrt{3}}{4}(1+\frac{2\sqrt{3}}{3}\sin\alpha)^2$.
This leads to a contradiction.

{\textit Subcase 2.3}: $0<a_{2}\leq \frac{3-\sqrt{3}}{2}\cdot(1+\frac{2\sqrt{3}}{3}\sin\alpha-a_{1})$.

We place $\triangle_{1}$ as in Fig.~\ref{Fig.17}(c).
The remaining equilateral triangles are used for the covering of the trapezoid
$Z^{\ast}_{12}=R^{\alpha}\setminus \triangle'_{1}$ with base lengths
$1+\frac{\sqrt{3}}{3}\sin\alpha-a_{1}-\cos\alpha$ and $1+\frac{2\sqrt{3}}{3}\sin\alpha-a_{1}$ and height $\sin\alpha$.
By Lemma~\ref{lem:2}, we deduce that if $R^{\alpha}$ cannot be covered, then
\begin{align*}
	\sum A(\triangle_{n})
	&< \frac{\sqrt{3}}{4}a_{1}^{2}+\frac{1}{2}\sin\alpha(2+\sqrt{3}\sin\alpha-2a_{1}-\cos\alpha)\\
	&~~~+\frac{\sqrt{3}}{2}a_{2}(1+\frac{\sqrt{3}}{3}\sin\alpha-a_{1}-\cos\alpha)+\frac{\sqrt{3}}{4}a_{2}\sin\alpha(\sqrt{3}+\cot\alpha).
\end{align*}

The maximum of the right-hand side is
$\frac{\sqrt{3}}{6}\sin^{2}\alpha-\frac{1}{2}\sin\alpha\cos\alpha+\frac{11-\sqrt{3}}{8}\sin\alpha+\frac{3(\sqrt{3}-1)}{8}(2-\cos\alpha)$ provided that the feasible region is given by $\frac{2\sqrt{3}}{3}\sin\alpha \leq a_{1}<1 + \frac{\sqrt{3}}{3}\sin\alpha-\cos\alpha$, $0<a_{2}\leq\frac{3-\sqrt{3}}{2}\cdot(1+\frac{2\sqrt{3}}{3}\sin\alpha-a_{1})$ and $a_{1}\geq a_{2}$.
This maximum, which is attained at $a_{1}=\frac{2\sqrt{3}}{3}\sin\alpha$ and $a_{2}=\frac{3-\sqrt{3}}{2}$, is less than $\frac{\sqrt{3}}{4}(1+\frac{2\sqrt{3}}{3}\sin\alpha)^2$.
This leads to a contradiction.

{\textit Case 3}: $0<a_{1}<\frac{2\sqrt{3}}{3}\sin\alpha$.

The covering construction is the same as that in Case~3 of Theorem~\ref{thm:2}, and the same right-hand side expressions are obtained in the corresponding subcases.
For each subcase, we recheck the maximum of the corresponding expression over the closure of its admissible domain for the present range of $\alpha$ and verify that it is still bounded above by $\frac{\sqrt{3}}{4}(1+\frac{2\sqrt{3}}{3}\sin\alpha)^2$.
Thus the same contradictions follow in Case~3.

By the discussions above we know that $\varrho({R}^{\alpha}, \triangle)\leq \frac{\sqrt{3}(\sqrt{3}+2\sin\alpha)^{2}}{12\sin\alpha}$.
However, any equilateral triangle with side length less than $1+\frac{2\sqrt{3}}{3}\sin\alpha$ cannot parallel cover ${R}^{\alpha}$, hence $\varrho({R}^{\alpha}, \triangle)\geq \frac{\sqrt{3}(\sqrt{3}+2\sin\alpha)^{2}}{12\sin\alpha}$.
As a consequence, $\varrho({R}^{\alpha}, \triangle)=\frac{\sqrt{3}(\sqrt{3}+2\sin\alpha)^{2}}{12\sin\alpha}$.
\end{proof}

\vspace{0.8cm}
\noindent{\bf Funding}. This research was supported by the National Natural Science
Foundation of China (Grant Number 12371326).

\end{document}